\documentclass[11pt]{amsart}
\usepackage[margin=0.9in]{geometry}

\usepackage[utf8]{inputenc}
\usepackage{amsmath,amssymb,amsthm,amscd,amsfonts,mathrsfs,verbatim,stmaryrd,graphicx}
\usepackage{float}
\usepackage{hyperref}
\usepackage{lmodern}

\usepackage[noabbrev,nameinlink,capitalise]{cleveref}
\usepackage{float}

\newtheorem{theorem}{Theorem}[section]
\newtheorem{thm}[theorem]{Theorem}
\crefalias{thm}{theorem}

\newtheorem{proposition}[theorem]{Proposition}
\newtheorem{prop}[theorem]{Proposition}
\crefalias{prop}{proposition}

\newtheorem{lm}[theorem]{Lemma}
\crefalias{lm}{lemma}

\newtheorem{corollary}[theorem]{Corollary}

\newtheorem{con}[theorem]{Conjecture}

\theoremstyle{remark}
\newtheorem{remark}[theorem]{Remark}

\theoremstyle{definition}
\newtheorem{de}[theorem]{Definition}

\def\N{\mathbb N}
\def\Z{\mathbb Z}

\def\R{\mathbb R}
\def\C{\mathbb C}

\def\Irr{\mathrm{Irr}}
\def\Core{\mathrm{Core}}
\def\om{\omega}
\def\Om{\Omega}
\def\a{\alpha}
\def\b{\beta}
\def\B{\mathbb B}
\def\k{\mathbb K}
\def\F{F}
\def\FF{\mathscr{F}}
\def\Fi{\varphi}
\def\fun{\rightarrow}
\def\Hilb{\mathrm{Hilb}}
\def\M{\mathcal{M}}
\def\MM{\mathscr{M}}
\def\P{\mathcal{P}}
\def\ini{\mathrm{ini}}

\def\iso{\cong}
\def\min{\mathrm{min}}
\def\max{\mathrm{max}}
\def\rk{\mathrm{rk}}
\def\H{h}
\def\CC{\C^*}
\def\la{\langle}
\def\ra{\rangle}
\def\CP{\mathbb{C}P}
\def\d{\operatorname{d}}
\def\Proj{\operatorname{Proj}}
\def\Tot{\operatorname{Tot}}
\def\codim{\operatorname{codim}}
\def\Pa{Painlev\'e}

\def\ph{pseudoholomorphic}

\def\MB{Morse--Bott}
\def\MBF{Morse--Bott--Floer}
\def\PW{$P=W$}

\def\llambda{t}

\DeclareMathOperator{\Gr}{Gr}

\usepackage{xcolor}
\definecolor{OxBlue}{HTML}{002147}

\hypersetup{colorlinks = true, citecolor={OxBlue}, urlcolor  ={OxBlue}, linkcolor={OxBlue} }

\title[Floer-theoretic filtration on Painlev\'e Hitchin systems]
{Floer-theoretic filtration on Painlev\'e Hitchin systems}
\author{Szilárd Szabó}
\address{Szilárd Szabó, Institute of Mathematics, Faculty of Science, E\"otv\"os Lor\'and University, P\'azm\'any P\'eter s\'et\'any 1/C, Budapest H-1117 and HUN-REN Alfr\'ed R\'enyi Institute of Mathematics, Re\'altanoda u. 13-15., Budapest H-1053,  Hungary}
\email{szabo.szilard@renyi.hu, szilard.szabo@ttk.elte.hu}
\author{Filip Živanović}
\address{ F. T. Živanović, 
Simons Center for Geometry and Physics, 
Stony Brook, NY 11794-3636, U.S.A.}
\email{fzivanovic@scgp.stonybrook.edu} 

\begin{document}

\begin{abstract}
We classify equivariant $\C^*$-actions on moduli spaces of Higgs bundles corresponding to the Painlev\'e equations. 
Using this, we compute the Floer-theoretic filtrations on the cohomology of these spaces, introduced by Ritter and the second author in \cite{RZ1}.
We compare it with the ``$P=W$'' and the filtration obtained by multiplicities of the irreducible components of the nilpotent cone, ultimately deducing that the Floer-theoretic filtration coincides with the multiplicity filtration, for all 2-dimensional Higgs moduli.
\end{abstract}

\maketitle
\setcounter{secnumdepth}{3}
\setcounter{tocdepth}{1}


\section{Introduction}

Hyperk\"ahler moduli spaces of Higgs bundles on smooth compact curves, possibly with parabolic\footnote{By parabolic Higgs bundles, we will always mean that the Higgs field has logarithmic poles compatible with a quasi-parabolic filtration.} or irregular singularities, have recently seen numerous applications in various contexts in Mathematics and Physics. 
In the compact (i.e., non-punctured) case, many such applications are made possible by the existence of a natural action of the multiplicative group $\C^*$ on the moduli spaces $\M$, by rescaling the Higgs field. 
Furthermore, this action is equivariant with respect to the Hitchin map, defined as the characteristic polynomial of the Higgs field
\begin{equation}\label{intro:hitchin map}
    \H:\M\fun \mathcal{B}\iso \C^{\frac{1}{2} \dim_\C \M}.
\end{equation}
The construction of this equivariant action does not generalize directly to all moduli spaces of parabolic or irregular Higgs bundles, simply because then the boundary conditions at the punctures are not preserved. 
However, there is a trick that sometimes allows one to get around this obstacle: namely, rescale both the fiber (i.e. spectral) and base (curve) variables. 
The general procedure was described in~\cite[Section~3]{Neitzke_Fredrickson}. 
Often, the existence of such an action carries important geometric or physical information~\cite{Fredrickson_Pei_Yan_Ye},~\cite{DGNPY}. 

In this article, we study equivariant holomorphic $\C^*$-actions on $2$-dimensional\footnote{Dimensions, unless otherwise stated, are meant to be over the complex field.} moduli spaces of Higgs bundles over curves. 
One set of examples corresponds to the {\Pa} equations $PI,\ldots , PVI$ (see Section~\ref{sec:Painleve}). 
Extensive literature is devoted to these spaces from various perspectives. 
Here, we give a necessary and sufficient condition for the existence of a $\C^*$-action on these spaces $\mathcal{M}^{PX}$, 
and spell out explicitly the construction of such an action whenever it exists.
Furthermore, we compute its weights on the Hitchin base $\mathcal{B}^{PX}$ and on the symplectic $2$-form $\Omega_I^{PX}$ on $\M^{PX}.$
 Another source of examples are moduli spaces $\M_\Gamma$ of parabolic Higgs bundles corresponding to the affine root systems $\widetilde{A}_0,\widetilde{D}_4, \widetilde{E}_6, \widetilde{E}_7, \widetilde{E}_8$, see for instance~\cite[Section~4.1]{Boalch_wild_character_varieties},~\cite[Proposition~8.1]{Oshima},~\cite{groechenig2014hilbert},~\cite{zhang2017multiplicativity}. 
 A degenerate case in this family corresponding to $\widetilde{A}_0$ is the cotangent bundle of an elliptic curve. 
 We carry out the same study for these spaces, too. 
The topology of the case $\widetilde{E}_6$ has been studied in detail in~\cite{Mihajlovic_thesis}. 

As an example of the general theory developed in~\cite{RZ1} by using Floer-theoretic techniques, the constructed $\CC$-actions endow the singular cohomology of these spaces with the structure of a filtration $\FF_{\bullet}$.
We review key features of this construction in Section~\ref{sec:Floer}. 
Given a label $X$ in the {\Pa} classification, we introduce a two-variable generating Puiseux polynomial $P^{PX}$ that encodes the non-trivial weights of $\FF$ on the cohomology of the corresponding moduli space $\mathcal{M}^{PX}$ along with their dimensions: 
\[
    P^{PX} (q,t) = \sum_{k\in\Z_{\geq 0}, \lambda\in\mathbb{Q}} \dim_{\mathbb{B}} \Gr_{\lambda}^{\FF} H^k(\mathcal{M}^{PX};\mathbb{B}) q^{\lambda} t^k, 
\]
where $\mathbb{B}$ is any base field of characteristic zero.
Our first result may be summarized as: 

\begin{theorem}
 The equivariant $\C^*$-actions on Hitchin moduli spaces in the {\Pa} cases have invariants given in Table~\ref{table}. 
\end{theorem}
\begin{table}[H] 
\label{table}
\begin{center}
\begin{tabular}{|c|c|c|c|}
 \hline 
 
 $X$ &  $w(\Omega_I^{PX} )$ & $w(\mathcal{B}^{PX})$ & $P^{PX} (q,t)$   \\
 \hline 
 \hline 
 $VI$  & 1 & 2 & $q^2 t^0 + (4 q^{\frac 12} + q^1) t^2$ \\
 \hline 
 $V$ & $\varnothing$ & $\varnothing$ &  $\varnothing$  \\ 
 \hline 
 $IV$ & 2 & 3 & $q^1 t^0 + 3 q^{\frac 13} t^2$ \\
 \hline 
 $III$ & $\varnothing$ & $\varnothing$ & $\varnothing$ \\ 
 \hline 
 $II$ & 3 & 4 & $q^{\frac 12} t^0 + 2 q^{\frac 14} t^2$ \\ 
 \hline 
 $I$ & 5  & 6  & $q^{\frac 13} t^0 + q^{\frac 16} t^2$ \\
 \hline 
\end{tabular}
\end{center}
\caption{The first two columns show the weights of the $\C^*$-action on the Dolbeault holomorphic symplectic form and on the Hitchin base, respectively. The last column shows the corresponding generating polynomial.}
\end{table}
Similar actions from a little bit different perspective are constructed in~\cite[Appendix~C]{Fredrickson_Pei_Yan_Ye}. 
The main difference is that our approach comes with a complete geometric understanding of the underlying moduli spaces, allowing us to work out the full Morse theory picture. 
On the other hand,~\cite[Appendix~C]{Fredrickson_Pei_Yan_Ye} works for higher-dimensional moduli spaces of irregular Higgs bundles of rank $2$ in genus $0$ too. 
We hope that the combination of these two approaches can bring new insights into the topology of such moduli spaces. 
A potential application of our result is towards Mirror Symmetry. 
Indeed, using the terminology of~\cite[Section~5.5]{KW}, the flowlines of our $\CC$-actions are branes of type $(B,A,A)$, that are expected to have duals of type $(B,B,B)$. \\

The computation of filtration $\FF$ was done in \cite{RZ2} for the
(aforementioned) 2-dimensional parabolic Higgs moduli $\M_\Gamma$,  
where $\Gamma=\{0,\Z/2,\Z/3,\Z/4,\Z/6\}$. 
To groups $\Gamma$, one associates affine Dynkin graphs
$Q_{\Gamma}:=\widetilde{A}_0, \widetilde{D}_4, \widetilde{E}_6, \widetilde{E}_7, \widetilde{E}_8$, respectively. The intersection between these and {\Pa} spaces is precisely $\M_{\Z/2}=\M^{PVI}$.
It has been shown that:

\begin{prop}\cite{RZ2} \label{Intro:Floer filter for parabolic higgs dim=2}
    The Floer-theoretic filtration $\FF(H^2(\M_\Gamma))$ refines the {\PW} filtration. Moreover, its ranks
    correspond to the labels of the imaginary root of the graph $Q_\Gamma.$
\end{prop}

By {\PW} we mean the perverse `P' filtration associated to the Hitchin map \eqref{intro:hitchin map}, which is in general proved to be equal to the weight `W' filtration of the corresponding character variety $\M'$, diffeomorphic to $\M$, under the non-Abelian Hodge correspondence, hence the ``\PW'' name. 

The labelling of the imaginary root mentioned in \cref{Intro:Floer filter for parabolic higgs dim=2}
can be described in terms of the multiplicities of the irreducible components of the central fiber $h_\Gamma^{-1}(0)$ 
of the Hitchin map
$$\H_\Gamma:\M_\Gamma\fun \mathcal{B}_\Gamma \iso \C,$$ with its inherited scheme structure.
That motivates us to define \textit{multiplicity filtration}, for a general Higgs moduli space $\M$. Given the central fiber of its Hitchin map 
with its irreducible components $E_i$ and their multiplicities $m_i,$
$$h^{-1}(0)=\sum_i m_i E_i,$$
denote $m:=\max_i\, m_i,$ and define the \textbf{multiplicity filtration} on $H^{\dim_\C \M}(\M)$ to be
    $$\MM_{\lambda}:=\la [E_i] \mid m_i/m \leq \lambda \ra.$$
This is well defined, as all components $E_i$ are equidimensional, of real dimension $\dim_\C \M$, and the central fiber is deformation-retract of $\M,$
thus $H^{\dim_\C \M}(\M)$ is generated by $[E_i]$.

\begin{remark}
Interestingly, the multiplicities $m_i$ are discussed for the
irreducible components $E_i$ of a special type (called very stable Higgs bundles) in recent work by Hausel--Hitchin \cite{hausel2022very}, where they show that these numbers
show up in the (conjectured) mirror symmetry for Higgs moduli spaces.
\end{remark}
With this notation in mind, summarizing the work of this paper and \cref{Intro:Floer filter for parabolic higgs dim=2}, we get:

\begin{thm} Given a $2$-dimensional Higgs moduli space $\M$ which has a 
holomorphic $\CC$-action and such that the Hitchin map has generically smooth fibers, its multiplicity and Floer-theoretic filtration coincide rank-wise
$$\rk\, \MM_\lambda = \rk\,\FF_\lambda.$$
When $\M$ is a parabolic Higgs moduli, they both refine the {\PW} filtration, whereas 
for {\Pa} spaces, {\PW} refines the former two filtrations. 
Thus, in the intersection, i.e. for {\Pa} VI space, these three filtrations coincide. 
\end{thm}

Thus, the last theorem compares the three filtrations for all 2-dimensional Higgs moduli for which filtrations $\FF$ and $\MM$ make sense.
An interesting avenue for future research is to understand how this generalises in higher dimensions. Due to \cite{groechenig2014hilbert}, we know that the Hilbert schemes of parabolic 2-dimensional moduli, 
$\Hilb^n(\M_\Gamma)$, are parabolic Higgs moduli spaces as well.
It would be interesting to understand the
relation between the {\PW}, Floer-theoretic, and multiplicity filtration for these spaces.\\\\
\noindent \textbf{Acknowledgments.} 
We thank Alexandre Minets
for helpful conversations, in particular for pointing out the coincidence between the multiplicities of the components of the central fiber and the corresponding imaginary root, for 2-dimensional parabolic Higgs moduli.
This work grew out of discussions at the workshop "Birational Geometry and Quantum Invariants" held at the Simons Center for Geometry and Physics in November 2023. 
The first author would like to thank the organizers of the meeting for the invitation, the Center for its hospitality, and grants KKP 144148 and K146401 of the agency NKFIH (Hungary) for support. The second author wishes to thank the Simons Center for its hospitality.

\section{Filtration from the Floer theory of a $\CC$-action}\label{sec:Floer}

In this section, we recall the relevant results from the construction of the Floer-theoretic filtration by Ritter--\v{Z}ivanovi\'c, which we are going to use in order to compute this filtration in the next section. To keep the current paper brief, we do not include many details regarding the Floer theory, but rather we refer the interested reader to their papers \cite{RZ1, RZ2}.

In that work, the authors consider 
connected symplectic manifolds $(Y,\omega)$
admitting a \textbf{contracting} $\CC$-action $\Fi$, which means that given an arbitrary point $y\in Y$, convergence point
\begin{equation}\label{equation_convergence_point}
     y_0:=\lim_{\C^* \ni t\rightarrow 0} t\cdot y \;\in Y^{\CC}
\end{equation}
always exists, and the set of these points, equal to the set of fixed points $\F:= Y^{\CC}$, is compact.
The $\CC$-action is $I$-{\ph} with respect to an almost complex structure $I$ compatible\footnote{Meaning that $\om(\cdot, I \cdot)$ gives a Riemannian metric.} with symplectic form $\omega.$
Furthermore, the $S^1$-part has to be Hamiltonian, meaning that there is a moment map, i.e. a function $$H:Y\fun \R$$ satisfying
$\om(\cdot, X_{S^1})=dH,$
where $X_{S^1}$ is the vector-field of the $S^1$-action.
Altogether, this is what is called a \textbf{symplectic $\CC$-manifold,} which we assume for $Y$ from now on.
Such manifolds have an important compact subset called the \textbf{core}
\begin{equation} \label{equation_of_the_core}
    \Core(Y):=\{ y\in Y \mid y_{\infty}:= \lim_{\C^* \ni t\rightarrow \infty} t\cdot y \text{ exists}\} \supset \F
\end{equation}
In the cases when $Y$ is an algebraic variety and $\CC$-action is algebraic, $\Core(Y)$ is a subvariety and a deformation retract of $Y,$ in particular 
\begin{equation}\label{equation cohomology of the core equal to the total space}
   H^*(Y)\iso H^*(\Core(Y)).
\end{equation}
We have an easy lemma that applies to the spaces considered in this paper.
\begin{lm}\label{lemma_core_is_the_central_fibre} 
Given a symplectic $\CC$-manifold $Y$ with a proper $\CC$-equivariant map $$\Psi: Y\fun \mathcal{B}$$
to an affine space $\mathcal{B}$ with a linear $\CC$-action with only positive weights,
then $\Core(Y)=\Psi^{-1}(0).$
\end{lm}
Consider a directed system of Hamiltonian Floer cohomologies $\{HF^*(H_\lambda)\}_{\lambda \in \R\setminus \P}$ with Hamiltonians $H_\lambda$ that are, 
outside a compact set, linear with slope $\lambda$ with respect to $H$, and
$\lambda$ does not belong to the (discrete) set of periods $\P$ of the $S^1$-action $\Fi$.

Briefly, $HF^*(H_\lambda)$ is the homology of the chain complex generated by 1-periodic orbits of the Hamiltonian vector-field $X_{H_\lambda}$ (defined as the dual of the differential,
$\om(\cdot, X_{H_\lambda}) = dH_\lambda$)
and the differential counts pseudo-holomorphic curves perturbed by the same vector field,
so satisfying the equation $\partial_s u + I (\partial_t u - X_{H_\lambda})=0.$ The cohomological gradings are by certain Robbin--Salamon indices, which in general are hard to compute. But in our setup, since the $S^1$-action is pseudoholomorphic, it amounts to 
computing winding numbers of paths of unitary matrices, which is computable, as we will see 
(Equations \eqref{Floer_grading_of_F_alphas_in_HF*(H_lambda)}, \eqref{CZindecesOfMorseBottSumbanifolds}).

The authors in \cite{RZ1} define the filtration by ideals on quantum cohomology ring $QH^*(Y)$: 
\begin{equation} \label{filtration_definition}
    \FF_\lambda^{\Fi}:=\ker(c_{\lambda}^*:QH^*(Y) \fun HF^*(H_\lambda))
\end{equation}
where $c_\lambda^*$ is a Floer-theoretic type of map (called a \textit{continuation map}).
For $p\in \P,$ simply define $\displaystyle \FF_{p}^{\Fi}:=\cap_{p< \lambda} \FF_{\lambda}^{\Fi}.$ 
Quantum cohomology as a $\k$-vector space is $QH^*(Y)=H^*(Y,\k)$, with
\begin{equation}\label{EqnSec2NovikovField}
\textstyle \k := \{\sum n_j T^{a_j}\mid a_j\in \R, a_j\to \infty, n_j\in \mathbb{B} \},
\end{equation}
is the \textit{Novikov field}, with
formal variable $T$ in grading zero, and $\mathbb{B}$ is any choice of base field of characteristic zero.\footnote{This condition is not necessary for the theory, but it is assumed in the computations, to avoid dealing with torsion coefficients.} Therefore, by ``letting $T\fun 0$'', one defines the \textit{specialized} $\Fi$-filtration on ordinary cohomology ring $H^*(Y):=H^*(Y,\B)$ by \textit{cup-ideals}
\begin{equation}\label{Specialised filtration}
    \FF_{\B,\lambda}^{\Fi}:=\ini(\FF_{\lambda}^{\Fi}) \subset H^*(Y;\mathbb{B}),
\end{equation}
where $\ini(x)$ is the initial term
$n_0 \in H^*(Y;\mathbb{B})$ of an element 
$x=n_0 T^{a_0} + (T^{>a_0}\textrm{-terms})$.

Filtration \eqref{Specialised filtration} is thus more feasible than \eqref{filtration_definition} to be comparable to other known filtrations, as we will see in the examples.
However, in practice, we will often blur the difference between them in practice as (by linear algebra) we have
 $$\rk_{\B}\, \FF_{\B,\lambda}^{\Fi}
    = 
    \rk_{\k}\, \FF_{\lambda}^{\Fi}.$$

These filtrations are compatible with cohomological grading, so one can consider them degree-wise.

In order for \eqref{filtration_definition} to be well defined, it is further assumed that there is a proper map with certain properties $$\Psi: Y\fun \mathcal{B}$$ to a symplectically-convex base $\mathcal{B}$, calling such $Y$ a \textbf{symplectic $\CC$-manifold over a convex base}. The obtained filtration is an invariant of such manifolds, and in general it \textit{depends} on the action $\Fi.$
A particular instance, which is enough for the spaces considered in this paper, 
is 
\begin{thm} Let $Y$ be a symplectic $\CC$-manifold
with a proper $(I,i)$-{\ph} map $$\Psi: Y \fun \mathcal{B}=\C^n$$
that is $\C^*$-equivariant with respect to a linear action on $\mathcal{B}$ with positive weights.
Then \eqref{filtration_definition}, thus \eqref{Specialised filtration} is well-defined.
\end{thm}

\cite{RZ1} provides many methods to obtain information on filtration $\FF^{\Fi}$, which are sometimes sufficient to completely determine it, as for the spaces considered in this paper.
We will briefly describe some of these methods.
Consider the connected components $\F_\a$ of the fixed locus of $\CC$-action $\Fi$:
\begin{equation}\label{fixed_locus}
Y^{\C^*}=\F=\sqcup_\a \F_\a.
\end{equation}
Each component has an $I$-linear weight decomposition of its tangent bundle,
\begin{equation}\label{weight decomposition}
  T\F_\a= \oplus_{k\in\Z} H_k,\ H_k := \{v \mid t\cdot v= t^k v\}. 
\end{equation}
For the spaces considered in this paper, these decompositions
are 
easily-computable, due to:\footnote{for proof see e.g. \cite[Lemma 8.7]{RZ1}}
\begin{lm}\label{omega_C_pairing}
Fix a component $\F_\a$. 
If $Y$ has an $I$-holomorphic symplectic structure $\Om_I$ with positive $\CC$-weight,
$t \cdot \Om_I=t^s \Om_I$ for some integer $s>0,$ then there is a non-degenerate pairing
$$\Om_I: H_k \times H_{s-k}\fun \C$$
in particular $H_k \iso H_{s-k}^{\vee}.$
\end{lm}

Cohomology of the total space and fixed loci are related due to \cite[Lemmas 3.5, 3.22]{RZ1}:
\begin{prop}\label{Morse_Bott_decomposition_theorem_for_moment_map} The $S^1$-moment map $H$ is a {\MB} function, and one has a decomposition\footnote{where the shift notation is usual -- given a graded module $A^*,$ the graded
module $A^*[k]$ is defined by $A[k]^i := A^{i+k}.$}
\begin{equation}\label{EqnFrankelIntro}
H^*(Y)\cong \oplus_{\a} H^*(\F_\a)[-\mu_\a],
\end{equation}
where $\mu_\a$ are {\MB} indices of $\F_\a,$ or equivalently $\mu_\a=2 \dim_\C (\oplus_{k<0} H_k)$.
In particular, among the fixed components $\F_\a$, there is precisely \textit{one} $\F_\min$
on which the moment map $H$ attains its minimum.
\end{prop}


In the setup of \cref{omega_C_pairing}, when $s=1$, the minimum $\F_{min}$
is an exact Lagrangian submanifold with respect to the Liouville forms $\mathbb{R}e(\Om_I),\mathbb{I}m(\Om_I)$, see
\cite{vzivanovic2022exact}.

Fixing an action $\Fi$, we will omit it in the following notation for filtration $\FF_{\lambda}:=\FF_{\lambda}^{\Fi}$.

The following statement tells us until when the filtration does not intersect the $H^*(\F_{\min}).$

\begin{prop}
\label{Cor intro about Fmin surviving}
Let $\lambda_{\min}:=1/(\textrm{maximal absolute weight of }\F_{\min})$.
If $H^{odd}(Y)=0$, and $\lambda<\lambda_{\min}$, then $c_{\lambda}^*:H^*(\F_{\min};\k)\to HF^*(H_{\lambda})$ is injective, so $\FF_{\lambda}^{}\cap H^*(\F_{\min};\k)=\{0\}$, and in \eqref{EqnFrankelIntro}
$$
\FF_{\B,\lambda}^{} \subset \oplus_{\a\neq \min} H^*(\F_{\a})[-\mu_\a]. 
$$

\end{prop}

Moreover, there is a statement describing when the unit
$1\in H^0(\F_{\min})=H^0(Y)$ enters the filtration (i.e. smallest $\lambda$ 
such that $1\in \FF_\lambda$), in the setup of \cref{omega_C_pairing}:

\begin{prop}\label{unit entering for weight-1 and weight-2 SHS} When $Y$ has an $I$-holomorphic symplectic structure $\Om_I$ of 
either 
\begin{enumerate}
    \item weight-1
    \item weight-2 and $\F_\min=\{\text{point} \}$
\end{enumerate}
the unit $1\in H^0(Y)$ enters the filtration \textit{precisely} at $\lambda=2$ for (1), and at $\lambda=1$ for (2).
\end{prop}

On the other hand, there are lower bounds for ranks of the filtration, which are readily computable, using only the Betti numbers $b_k (\F_\a)$ of the fixed loci, and their weight decompositions \eqref{weight decomposition}. Introduce Floer-theoretic indices
\begin{equation}\label{Floer_grading_of_F_alphas_in_HF*(H_lambda)}
\mu_\lambda(\F_\a) := \dim_\C Y - \dim_\C \F_\a - \sum_{k} \dim_{\C}(H_k) W(\lambda k), 
\end{equation}
where 
\begin{equation}\label{WfunctionForRSAppendix}
W :\R \to \Z, \quad
	W(x) := \left\{
	\begin{array}{ll}
	2 \lfloor x \rfloor + 1 & \text{if} \ x \notin \Z \\
	2x & \text{if}  \ x \in  \Z.
	\end{array}
	\right. 
	\end{equation}
 is a step-function. We have the following:

\begin{prop}\label{lower_bounds_filtration}
$    \rk\, \FF_{\lambda} (H^k(Y)) \geq \sum_\a b_{k-\mu_\a} (\F_\a)-b_{k-\mu_\lambda(\F_\a)}(\F_\a).$
\end{prop}

The statements mentioned so far completely determine the filtration in the case of {\Pa} spaces. However, in order to determine the filtration in general, one has to use the {\MBF} spectral sequences defined in \cite{RZ2}. Let us briefly describe this method.
Given an integer $m>1$, observe 
the fixed locus $Y_m:=Y^{\Z/m}$ of the cyclic subgroup $\Z/m \leq \CC$.
It is an $I$-{\ph} symplectic submanifold that breaks into a finite number of connected components, out of which we label
by $\{Y_{m,\b}\}_{\b}$ the non-compact ones, or equivalently those that go out of the $\Core(Y).$\footnote{recall the definition of core in \eqref{equation_of_the_core}}
We call them \textbf{outer torsion submanifolds}, as they have torsion isotropy subgroups of the $\C^*$-action.
Denote their intersections with 
energy hypersurfaces by 
\begin{equation}\label{Morse_Bott_submanifolds_of_orbits}
    B_{T,\b}:=Y_{m,\b} \cap \{c'(H)=T\}.
\end{equation}
Here, $c=c(H)$ is a real convex function and
$T=k/m, k\in \N$ captures periods of orbits of the $S^1$-action.
These periods we call the \textbf{outer $S^1$-periods} as these $S^1$-orbits are outside of the core. 
From the construction of a specific family of \eqref{filtration_definition} Hamiltonians $H_\lambda$ in \cite{RZ2}, they conclude:
\begin{thm}\label{filtration_stability_theorem}
    $\FF_{\lambda}^{}=\FF_{\lambda'}^{}$ if there are no outer 
    $S^1$-periods in the interval $(\lambda,\lambda'].$
\end{thm}
In other words, $\FF_{\lambda}^{\Fi}$ can only jump when $\lambda$ passes through an outer $S^1$-period. 
Finally, we have a theorem on the aforementioned spectral sequence:
\begin{theorem}\label{Theorem spectral seq}
Let $Y$ be a symplectic $\C^*$-manifold over convex base satisfying $c_1(Y)=0.$
There is a spectral sequence of $\k$-modules 
$E_{r}^{pq}(\Fi) \Rightarrow 0$ 
whose 
$E_1$-page is
\begin{equation}\label{spectral sequence}
 E_{1}^{pq}(\Fi)= 
 \begin{cases}
				H^q(Y;\k) & \text{if } p=0\textrm{ and }\\
				\bigoplus_{\b} H_{p,\b}^*[-\mu(B_{T_p,\b})] & \text{if } p \leq 0 \\ 
				0 & otherwise.
\end{cases}
\end{equation}
This spectral sequence determines the filtration: a class $x \in H^*(Y;\k)=QH^*(Y)$ 
lies in $\FF^{\Fi}_{\lambda}$ 
iff the columns of the spectral sequence 
having $T_p \leq \lambda$ 
annihilate $x\in E_1^{0,*}=H^*(Y;\k)$ via some differential $\operatorname{d}_r$ of the spectral sequence $E_{r}^{pq}(\Fi)$. 
\end{theorem}

In the spectral sequence above, the periods $T_p$ \textit{increase as $p$ decreases,}
and $\mu(B_{T_p,\b})$ are again certain Floer-theoretic indices, which we can compute explicitly:
\begin{equation}\label{CZindecesOfMorseBottSumbanifolds}\textstyle
\begin{split}
  \mu(B_{T_p,\b}) & = \textstyle \dim_{\C} Y - \frac{1}{2}\dim_{\R} B_{T_p,\b}-\frac{1}{2}- \sum_i W(T_p w_i)\\ & \textstyle = \codim_{\C}\, Y_{m,\b} - \sum_i W(T_p w_i),
\end{split}
\end{equation}
 where $w_i$ are the weights 
of the convergence point\footnote{recall \eqref{equation_convergence_point}} $y_{0}$ of an arbitrary $y\in B_{T_p,\b}.$ 
In general, there is another Floer-theoretic local \textit{energy spectral sequence} 
$$\oplus_\b H^*(B_{T_p,\b};\k)[-\mu(B_{T_p,\b})] \Rightarrow \oplus_\b H_{p,\b}^*[-\mu(B_{T_p,\b})]$$
which collapses in low dimensions $\dim_\C(Y)\leq 2,$ in particular for all spaces considered in this paper, so computing \eqref{spectral sequence} amounts to computing ordinary cohomologies $H^*(B_{T_p,\b};\k),$ and their shifts $\mu(B_{T_p,\b})$. 
In practice, one computes these by knowing the topology of fixed loci $\F_\a$ and their weight decompositions \eqref{weight decomposition}, as each $Y_{m,\b}$ 
is obtained via the $\C^*$-flow from (one or more) fixed $\F_\a,$ where it converges (when $t\fun 0$) to the ($\Z/m$-fixed) part $$T_{\F_\a} Y_{m,\b} =\oplus_{b\in \Z} H_{mb}$$ 
of the weight decomposition. From this information, one can compute 
$$H^*(Y_{m,\b})=\oplus_{\F_\a \subset Y_{m,\b}} H^*({\F_\a})[-\mu_\a']$$ 
from the same formula as \eqref{EqnFrankelIntro}, where $\mu_\a'$ are the {\MB} indices of $\F_\a$ under the restriction $H|_{Y_{m,\b}}$, or equivalently 
dimension of negative weight space of $\F_\a$ within $Y_{m,\b}$.
Then, the cohomology of the hypersurface $B_{T_p,\b}$ can be computed
in terms of the cohomology of $Y_{m,\b}$ and its intersection form by \cite[Proposition 6.14]{RZ2}:\footnote{The conditions of the ambient cohomology support up to mid-degree can be lifted, with a caveat that we then know cohomology groups of hypersurface up to some mid-degree part, which gets bigger with the top degree of support.}

\begin{prop}\label{hypersurface_cohomology}
Assuming that $H^*(Y)$ is supported in degrees $*\leq n=\dim_\C Y,$ 
the cohomology of energy hypersurface $B=\{H=c\}\subset Y$ of a proper function $H$ and $c \gg 0$ 
is given by 

$$ 
                H^k(B) \iso    \begin{cases}
			H^{2n-1-k}(Y) & \ k\geq n+1,\\
            H^{n-1}(Y)\oplus \ker \Fi  & \ k=n,\ n-1,\\
			H^k(Y) & \ k\leq n-2. 
                          \end{cases}
$$
where $\Fi:H_n(Y)\times H_n(Y) \fun \k$ is the intersection form. 
\end{prop}

\begin{remark}
For $Y$ affine (e.g. character varieties), the condition of the proposition holds. 
When $\dim_{\C} Y = 2$, the lowest weight piece of the weight filtration in Deligne's Mixed Hodge Structure in these cohomological degrees has a similar description in terms of the kernel of the intersection form; see~\cite[Section~6.9]{ElZein_Nemethi},~\cite[Section~6]{Nemethi_Szabo}. 
\end{remark}

\section{Construction of $\C^*$-actions and filtration computation}\label{sec:Painleve}

{\Pa} Hitchin systems are certain complex $2$-dimensional integrable systems of stable irregular Higgs bundles of rank $2$ with underlying curve the complex projective line $\C P^1$. 
We will let $z$ stand for the standard coordinate on an affine $\C\subset \C P^1$. 
The notation $O(1)$ denotes polynomials of nonnegative degree in $z$. 
More generally, $O(z^k)$ denotes Laurent polynomials only containing terms of degree at least $k$ with respect to $z$. 
Throughout, we will denote by $K = \Omega_{\C P^1}^1$ the canonical line bundle of $\C P^1$, and by
\[
  \pi\colon\Tot(K)\to  \C P^1
\]
its total space. We let
\[
      \zeta \in H^0 (\Tot(K), \pi^* \Omega_{\C P^1}^1 )
\]
stand for the tautological Liouville $1$-form on the total space $\Tot(K)$. 
By the canonical inclusion of $1$-forms 
\[ 
  \pi^* \Omega_{\C P^1}^1 \subset \Omega_{\Tot(K)}^1 ,
\]
$\zeta$ may be thought of as a holomorphic $1$-form on $\Tot(K)$. 
Then, as it is well-known, 
\begin{equation}\label{eq:Liouville}
 \Omega_I = - \d \! \zeta \in H^0 (\Tot(K), \Omega_{\Tot(K)}^2 ) 
\end{equation}
is a non-degenerate closed holomorphic $2$-form on $\Tot(K)$.
In different terms, $\Omega_I$ defines a symplectic structure on $\Tot(K)$, called the canonical (or Liouville) symplectic structure. 

In order to define the {\Pa} Hitchin systems, one needs to choose further an effective divisor $D$ of length $4$ on $\C P^1$. 
We may and will choose it so that its point of support with the highest multiplicity is $z=0$. 
The {\Pa} Hitchin systems $\mathcal{M}^{PX}, X\in \{I, \ldots , VI\}$ arise as holomorphic symplectic moduli spaces of certain meromorphic Higgs bundles with polar divisor (bounded from above by) $D$, and with fixed polar part at $D$. 
The holomorphic symplectic form is called Atiyah--Bott $2$-form, and denoted by $\Omega_I^{PX}$. 
Partitions of $4$ and a further choice of the type of the polar part (semi-simple or with non-trivial nilpotent part) at each point of $D$ lead to various possibilities commonly referred to as the {\Pa} I, II, III, IV, V and VI cases (with some further labels attached to the cases II, III, IV and V, depending on the authors). 
We will denote the corresponding {\Pa} Hitchin systems by $\mathcal{M}^{PI},\ldots, \mathcal{M}^{PVI}$. 
According to~\cite[Theorem~A]{KLSSz}, for each $X$ there is a birational map 
\[
    f_X\colon  \mathcal{M}^{PX} \to \Tot(K)
\]
such that\footnote{Strictly speaking, this was done in the de Rham complex structure rather than the Dolbeault structure needed here; however, a straightforward modification of that proof gives the Dolbeault case too. See also \cite[Proposition~7.12]{Markman_Spectral_curves} for a related result in the Dolbeault case.}
\begin{equation}\label{eq:symplectomorphism}
    f_X^* \Omega_I = \Omega_I^{PX}. 
\end{equation}
In addition, they carry fibrations
\begin{equation}\label{eq:Hitchin_map}
 \H\colon \mathcal{M}^{PX}\to \mathcal{B}^{PX} \iso \C
\end{equation}
turning them into completely integrable systems whose generic fibers are abelian varieties. 
We construct equivariant $\C^*$-actions case by case in Sections~\ref{sec:PI}--\ref{sec:PVI}. 
In Section~\ref{sec:PIII-V} we show that in the cases not covered by our constructions, equivariant $\C^*$-actions do not exist. 
In Section~\ref{sec:E6} we treat the complex $2$-dimensional moduli spaces of rank $3,4$ and $6$ parabolic Higgs bundles over $\C P^1$ with three logarithmic singularities, corresponding to the root systems $\widetilde{E}_6, \widetilde{E}_7$ and $\widetilde{E}_8$ respectively. 
From~\eqref{eq:Liouville} and~\eqref{eq:symplectomorphism}, it follows immediately that: 
\begin{proposition}
 For every $X$ and any $\C^*$-action on $\mathcal{M}^{PX}$ of the form $(z, \zeta )\mapsto (z^{w(z)}, \zeta^{w(\zeta )})$, the weight of the Atiyah--Bott holomorphic symplectic form $\Omega_I^{PX}$ agrees with $w(\zeta )$. 
\end{proposition}

\subsection{Painlev\'e I case}\label{sec:PI}

Consider the divisor $D = 4 \cdot \{ 0 \}$ in $\C P^1$. 
The local section $\frac{\d\! z}{z^4}$ trivializes $K(D)$ over $\C$. 
In~\cite[Section~2.3.3]{ISS1} the following local form for a ramified irregular Higgs field was considered
{ \[  
    \theta = \left( \begin{pmatrix}
             b_{-8} & 1 \\
              0	& b_{-8}
            \end{pmatrix} z^{-4} 
            +
         \begin{pmatrix}
             0 & 0 \\
              b_{-7} & b_{-6}
            \end{pmatrix} z^{-3} 
               +
         \begin{pmatrix}
             0 & 0 \\
              b_{-5} & b_{-4}
            \end{pmatrix} z^{-2}
               +
         \begin{pmatrix}
             0 & 0 \\
              b_{-3} & b_{-2}
            \end{pmatrix} z^{-1} 
            + O(1)
            \right) \otimes \d\! z.
\] }
\hspace{-0.32cm} Here, $b_{-8}, \ldots b_{-3}\in\C$ are fixed, and $b_{-2}=0$ by the residue theorem. 
We must assume that $b_{-7} \neq 0$. 
For the sake of simplicity and concreteness, let us therefore fix 
\begin{equation}\label{eq:choices}
 b_{-7}=1, \qquad b_{-8} = b_{-6} = \cdots = b_{-2} = 0.
\end{equation}
According to \cite[Theorem~1.3]{ISS1}, the associated Hitchin system $\H:\mathcal{M}^{PI}\fun B^{PI}$ has a unique singular fiber of type $II$ in the Kodaira classification, namely a genus $0$ curve with one cuspidal point, having a local equation analytically equivalent to $y^2 = x^3$.
The above local form simplifies to 
\begin{equation}\label{eq:thetaPI}
 \theta = \begin{pmatrix}
             0 & z^{-4} \\
              z^{-3} & 0
            \end{pmatrix} \d\! z 
            + O(1) \d\! z.
\end{equation}
Consider its characteristic polynomial
\[
   \det ( \theta - \zeta ) = \zeta^2 + a_1 (z) \zeta + q(z),
\]
where $a_1\in H^0 (\C P^1, K (D)),$ $ q\in H^0 (\C P^1, K^{\otimes 2} (2D))$. 

\begin{proposition}\label{prop:PI}
For the choices~\eqref{eq:choices}, we have $a_1 = 0$ and
\[
 q(z) = ( z^{-7} + b z^{-4} ) (\d\! z)^{\otimes 2} 
\]
for some $b\in \C$. 
Conversely, for any $q$ of this form there exists $(\mathcal{E}, \theta)\in \mathcal{M}^{PI}$ satisfying 
\[
   \det ( \theta - \zeta ) = \zeta^2 + q(z). 
\]
\end{proposition}

\begin{proof}
 By~\eqref{eq:choices}, the trace $-a_1$ is a holomorphic differential over $\C P^1$ with no pole at $0$. 
 As it has no other poles either, we have $a_1 \in H^0 (\C P^1, K) = 0$. 

 To get a term of the determinant, the top right $z^{-4}$ entry of $\theta$ may multiply a bottom left entry of the additional $O(1)$ term to form an unspecified term $O(z^{-4})(\d\! z)^{\otimes 2}$. 
 It is easy to see that the coefficients of the terms of degrees smaller than $-4$ must be the ones stated. 
 Consider the set 
 \[
  \mathcal{B}^{PI} = \{q\colon \; q \mbox{ gives rise to a Higgs field satisfying~\eqref{eq:choices}} \} \subset H^0 (\C P^1, K^{\otimes 2} (2D)). 
 \]
 The difference between any two elements $q_1, q_2\in \mathcal{B}$ is a quadratic differential with a pole of order at most $4$ at $0$. 
 Since $q_1, q_2$ have no other poles, their difference belongs to $H^0(\C P^1, K^{\otimes 2} (D))$. 
 It follows that this set is an affine space modeled on 
 \[ 
  H^0(\C P^1, K^{\otimes 2} (D))\cong H^0(\C P^1, \mathcal{O}) = \C. 
 \]
 A generator of this cohomology space is $z^{-4}(\d\! z)^{\otimes 2}$. 
\end{proof}

The space $\mathcal{B}^{PI}$ will be referred to as \emph{irregular Hitchin base} in the {\Pa} I case. 
For any $b\in \C$, let us define the projective curve
\[
 \Sigma^{PI}_b = \{ (z,\zeta ) \colon \; \zeta^2 + ( z^{-7} + b z^{-4} ) (\d\! z)^{\otimes 2} = 0 \}. 
\]
in the Hirzebruch surface 
 $\Proj ((K(D))^{-1} \oplus \mathcal{O}).$
We will refer to $\Sigma^{PI}_b$ as the \emph{irregular spectral curve} associated to $(z^{-7} + b z^{-4})(\d\! z)^{\otimes 2} \in \mathcal{B}^{PI}$. 
The case $b=0$ gives the central fiber of type $II$. 

\begin{proposition}\label{prop:PI_action}
 There exists a holomorphic $\CC$-action 
 on $\mathcal{M}^{PI}$, equivariant with respect to the action $b\mapsto \llambda^6 b$ on $\mathcal{B}^{PI}$. 
 It acts on $\Omega_I^{PI}$ with weight $5$. 
\end{proposition}

\begin{proof}
Let $\C^*$ act with weights $w(z), w(\zeta)$ respectively: 
\[
 \llambda\cdot (z, \zeta ) = (\llambda^{w(z)} z, \llambda^{w(\zeta)} \zeta). 
\]
Our aim is to choose these weights so that for any $b\in\C$ the curve $\Sigma^{PI}_b$ gets transformed into a curve $\Sigma^{PI}_{b'}$ for some other $b'\in\C$. 
Note that 
\begin{align*}
  \zeta^2 & \mapsto \llambda^{2w(\zeta)} \zeta^2 \\
  z^{-7} (\d\! z)^{\otimes 2}  & \mapsto \llambda^{-5 w(z)} z^{-7} (\d\! z)^{\otimes 2}
\end{align*}
It is easy to see that 
$w(z) = - 2, w(\zeta ) = 5 $
solves the problem. 

Fixing any $(z, \zeta) \in \Sigma_b$ for some $b\in\C$ we now need to determine the value $b' = b'(b,\llambda )\in \C$ such that
\[
 (\llambda^{-2} z, \llambda^5 \zeta) \in \Sigma_{b'}
\]
holds for every $\llambda \in \C^*$. We are given that 
\begin{equation}\label{eq:spectral_curve_PI}
 \zeta^2 + ( z^{-7} + b z^{-4} ) (\d\! z)^{\otimes 2} = 0. 
\end{equation}
Now, with the above choices, we find 
\[
 z^{-4} (\d\! z)^{\otimes 2} \mapsto \llambda^{4} z^{-4} (\d\! z)^{\otimes 2}. 
\]
Multiplying~\eqref{eq:spectral_curve_PI} through by $\llambda^{10}$ we find 
\[
 (\llambda^5 \zeta)^2 + (\llambda^{-2} z)^{-7} (\d (\llambda^{-2}z))^{\otimes 2} +  \llambda^6 b (\llambda^{-2} z)^{-4} (\d (\llambda^{-2}z))^{\otimes 2} = 0 .
\]
Therefore, the sought action on $\mathcal{B}^{PI}$ reads as 
 $b\mapsto b' = \llambda^6 b.$ 
\end{proof}

Let us determine the $\CC$-fixed loci of $Y:=\M^{PI}$ and their weight decompositions.
Due to \cref{lemma_core_is_the_central_fibre}, 
$$\Core(Y)=\H^{-1}(0),$$ thus is a cuspidal genus $0$ curve. 
It contains, but is not equal\footnote{as the fixed locus is smooth} to the $\CC$-fixed locus.
Due to the cohomology rank count (combination of \eqref{equation cohomology of the core equal to the total space} and \eqref{EqnFrankelIntro}) 
the fixed locus consists of 2 points, $$Y^{\CC}=\F=\F_0 \sqcup \F_1.$$ One of these, say $\F_0$, has to be the cuspidal point, since
$\Core(Y)\setminus\F$ consists of smooth $\CC$-flowlines between $\F_0$ and $\F_1$.
One of the points is the minimum $F_{\min}$ and the other is the maximum $F_{\max}$ for the
restriction of the moment map $H:Y\fun \R$ of the action $S^1\leq \CC$ on the core. Now the decomposition
$$T_{F_{\max}} Y =H_{-k_1} \oplus H_{k_2} $$ has a negative and a positive weight,
where the negative corresponds to the $\CC$-flow from $F_{min}$, and the positive goes outside the core. 
Due to \cref{omega_C_pairing}, and the last statement of \cref{prop:PI_action}, $k_2-k_1=5,$ in particular $k_2\geq 6.$
Due to the first statement of \cref{prop:PI_action}, we have $k_2|6$, hence $k_2=6,$ and so $k_1=1.$
Thus, the singular point $F_0$ cannot be $F_{\max}$, as otherwise, in local coordinates $(x,y)$ around the $F_0=(0,0)$
core would be a (irreducible) 
cuspidal curve invariant under the $(t x,t^6 y)$ action, so of type $y=x^6,$ which is smooth, contradiction.
Thus, $F_1=F_{\max}$, and the $\CC$-flowline within the core is weight-1, hence not torsion (i.e. fixed under a cyclic subgroup of $\CC$). 
On the other hand, the $\CC$-flowline emanating from $F_1$, of weight $6$, represents the Hitchin section of $h$, consisting of irregular Higgs bundles of the form 
\[
    \theta = \begin{pmatrix}
    0 & z^{-4} \\
    z^{-3} + b  & 0 
    \end{pmatrix} 
    \operatorname{d}\! z
\]
on $\mathcal{O}_{\C P^1} (-1) \oplus \mathcal{O}_{\C P^1} (1)$, using the isomorphism $K_{\C P^1} (4\cdot \{ 0 \})\cong \mathcal{O}_{\C P^1} (2)$ that maps $z^{-4} \operatorname{d}\! z$ to the degree $2$ polynomial $0\cdot z^2 + 0\cdot z + 1$. 
Indeed, as one readily checks, the choice $b=0$ gives rise to a spectral curve that is non-singular and ramified over $z=0$, and has a cusp point over $z=\infty$, and then the Hitchin section intersects it in the ramification point $(z, \zeta ) = (0,0)$. 

The weight decomposition of minimum $F_0=F_{\min}$ 
$$T_{F_{\min}} Y =H_{l_1} \oplus H_{l_2} $$
has only positive weights, and as before, $l_1+l_2=5.$ Moreover, both of these weight spaces correspond to the $\CC$-flow which goes outside
of the core (as due to the equivariant tubular neighborhood theorem, their image under the exponential map has to be smooth),
hence as before $l_1|6,\ l_2|6.$ Thus, up to a permutation, $l_1=2,\ l_2=3$.

Computation of the indices \eqref{Floer_grading_of_F_alphas_in_HF*(H_lambda)} gives, where we denote $\mu_{\lambda+}:=\mu_{\lambda+\delta}$ for $1\gg\delta>0$,
$$\mu_{1/6+}(F_0)=\mu_{1/6+}(F_1)=0,\ \mu_{1/3+}(F_0)=\mu_{1/3+}(F_1)=-2,$$
and {\MB} indices are (real dimension of negative weight spaces) $$\mu_0=0,\ \ \mu_1=2.$$
Therefore, 
\cref{lower_bounds_filtration} gives the lower bounds
$$\rk\, \FF_{1/6}(H^2(Y)) \geq 1,\ \rk\, \FF_{1/3} \geq 2,$$
hence $\FF_{1/6}\supset H^2(Y)$ and $\FF_{1/3}=H^*(Y).$ On the other hand, due to \cref{Cor intro about Fmin surviving}, $1\in H^0(F_{\min})$ is not in the filtration until $\lambda=1/3,$
so, together with the previous, we have:
\begin{equation}\label{filtration_panleve_I}
    0\subset \FF_{1/6}=H^2(\M^{PI}) \subset \FF_{1/3} =H^*(\M^{PI}).
\end{equation}
The outer periods of the $S^1$-action are multiples of $1/2,1/3,$ and $1/6,$ hence due to \cref{filtration_stability_theorem}, \cref{filtration_panleve_I} shows exactly when the filtration jumps, which then determines the filtration completely.

We could also use the spectral sequence method to reach the same conclusion. Recall the weight decompositions 
$$T_{F_0} Y =H_2 \oplus H_3, \quad  \ T_{F_1} Y = H_{-1} \oplus H_6,$$ 
where only $H_{-1}$ flows within the core. 
The other weight spaces yield the outer torsion manifolds 
$$Y_6\iso \C, \quad  Y_3 \iso \C \sqcup \C, \quad  \ Y_2 \iso \C \sqcup \C$$
where $Y_6$ is the Hitchin section and the second components in $Y_3$ and $Y_2$ are equal to $Y_6,$ as $\Z/3\leq \Z/6,$ and $\Z/2\leq \Z/6.$
Thus, their intersections with the energy hypersurfaces \eqref{Morse_Bott_submanifolds_of_orbits} are 
$$B_{1/6} \iso S^1,\ B_{1/3} \iso S^1 \sqcup S^1,\ B_{1/2} \iso S^1 \sqcup S^1.$$
Their shifts, computed via \eqref{CZindecesOfMorseBottSumbanifolds}, give the columns of the $E_1$-page of the {\MBF} spectral sequence, shown on \cref{MI_sp_seq}. 
In it, we show all the differentials 
$\d_r,\ r\geq 1$ stapled together.
The column $H^*(B_1)$ is determined by \cref{hypersurface_cohomology}, as the intersection form of $\M^{PI}$
is trivial. Indeed, this is true as $H_2(\M^{PI})=[\Core(Y)],$ and its self-intersection is zero, being a fiber class (of the Hitchin map).
\begin{figure}[h]
		\centering
		{
			\includegraphics[scale=1]{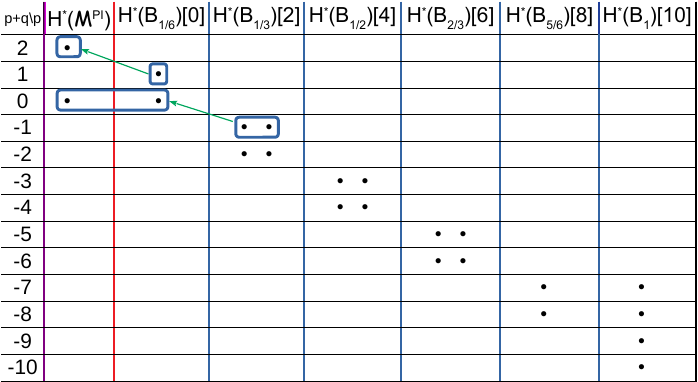}
                \caption{Spectral sequence for $X=I$}
                \label{MI_sp_seq}
		}
\end{figure}

Due to the last statement of
\cref{Theorem spectral seq}, the spectral sequence determines the filtration -- we see that the class in $H^2(\M^{PI})$ gets annihilated by the column with period $1/6$, hence lies in $\FF_{1/6}.$ Similarly, $H^0$ lies in $\FF_{1/3},$ hence we recover the filtration \eqref{filtration_panleve_I} indeed.

\subsection{Painlev\'e II case}

The divisor is again $D = 4 \cdot \{ 0 \}$. 
The local form of the Higgs field is 
\[
   \theta = \left[ z^{-4} \begin{pmatrix}
         a_+ + b_+ z + c_+ z^2 + \llambda_+ z^3 & 0 \\
         0 & a_- + b_- z + c_- z^2 + \llambda_- z^3 
        \end{pmatrix} + O(1)
  \right] \otimes \d\! z. 
\]

We fix any $a\in \C^*$ and choose 
\[
 a_{\pm} = \pm a, \quad b_{\pm} = c_{\pm} = \llambda_{\pm} = 0. 
\]
It follows from~\cite[Theorem~1.1]{ISS1} that the corresponding Hitchin moduli space $\mathcal{M}^{PII}$ has one singular fiber, of type $III$, namely consisting of two genus $0$ components intersecting in a single $A_3$ type curve singularity also known as a tacnode, with a local equation equivalent to $y^2 = x^4$. 

\begin{proposition}
 We have $\operatorname{tr} (\theta ) = 0$ and 
\[
 q(z) = ( -a^2 z^{-8} + b z^{-4} ) (\d\! z)^{\otimes 2} 
\]
for some $b\in \C$. 
Conversely, for any $q$ of this form there exists $(\mathcal{E}, \theta)\in \mathcal{M}^{PII}$ satisfying 
\[
   \det ( \theta - \zeta ) = \zeta^2 + q(z). 
\]
\end{proposition}

\begin{proof}
 Since $\theta$ has no singularity away from $z=0$ (including at $z=\infty$), it must be of the form 
\[
 \theta =  \begin{pmatrix}
         a z^{-4} + \alpha & \beta \\
         \gamma & - a z^{-4} + \delta 
         \end{pmatrix}
\]
for some $\alpha , \beta , \gamma , \delta \in \C$. 
Vanishing of the trace follows exactly as in Proposition~\ref{prop:PI}. 
It is easy to check that then the determinant reads as 
\[
 q(z) = \left( -a^2 z^{-8} - 2a \alpha z^{-4} + (-\alpha^2 - \beta \gamma)z^0 \right) (\d\! z)^{\otimes 2}. 
\]
The same argument as in Proposition~\ref{prop:PI} gives vanishing of the coefficient of $z^0$. 
\end{proof}

The space of quadratic differentials $q$ as in the Proposition will be called \emph{irregular Hitchin base} in the {\Pa} II case, and denoted by $\mathcal{B}^{PII}$. 
For any $b\in\C$ we define the spectral curve by the equation 
\begin{equation}\label{eq:spectral_curve_PII}
  \Sigma^{PII}_b = \{ (z,\zeta ) \colon \;  \zeta^2 +  \left( -a^2 z^{-8} + b z^{-4} \right) (\d\! z)^{\otimes 2} = 0 \}.
\end{equation}
The central fiber (of type $III$) again corresponds to $b=0$. 

\begin{proposition}\label{prop:PII_action}
 There exists a $I$-holomorphic $\CC$-action on $\mathcal{M}^{PII}$, equivariant with respect to the action $b\mapsto \llambda^4 b$ on $\mathcal{B}^{PII}$. 
 It acts on $\Omega_I^{PII}$ with weight $3$. 
\end{proposition}

\begin{proof}
 With notation as in Proposition~\ref{prop:PI_action}, the action on~\eqref{eq:spectral_curve_PII} is 
 \begin{equation}\label{eq:transformation_spectral_curve_II}
  \llambda^{2w(\zeta )} \zeta^2 - a^2 \llambda^{-6 w(z )} z^{-8} (\d\! z)^{\otimes 2} + b' \llambda^{-2 w(z )} z^{-4} (\d\! z)^{\otimes 2} = 0 
 \end{equation}
Thus, to preserve the family $\{ \Sigma^{PII}_b \}_{b\in\C}$ of spectral curves, we need to impose the condition 
 \[
  2w(\zeta ) = -6 w(z ). 
 \]
 The coprime solution is 
 \[
  w(\zeta ) = 3, \quad w(z ) = - 1. 
 \]
 We then see that in order for~\eqref{eq:transformation_spectral_curve_II} to 
 be the equation of spectral curve ~\eqref{eq:spectral_curve_PII}, we must have $b'= \llambda^4 b$. 
\end{proof}

\subsection{Painlev\'e II (Flaschka--Newell) case} \label{sec Pa II}

The local form of the irregular Higgs field at $z=0$ is 
\[
    \theta = \left[ \begin{pmatrix}
             b_{-6} & 1 \\
              0	& b_{-6}
            \end{pmatrix} z^{-3}
               +
         \begin{pmatrix}
             0 & 0 \\
              b_{-5} & b_{-4}
            \end{pmatrix} z^{-2}
               +
         \begin{pmatrix}
             0 & 0 \\
              b_{-3} & b_{-2}
            \end{pmatrix} z^{-1} 
            + O(1)
            \right] \otimes \d\! z, 
\]
and $\theta$ has a logarithmic pole with a nilpotent residue at $z=\infty$.  
We make the choices 
\begin{equation}\label{eq:choicesIIFN}
 b_{-5}=1, \qquad b_{-6} = \cdots = b_{-2} = 0.
\end{equation}
The corresponding Hitchin moduli space $\mathcal{M}^{PIIFN}$ again only has one singular fiber, of type $III$, see~~\cite[Theorem~2.7]{ISS3} and its complement~\cite[Lemma~2~(4)]{Sz_PW}. 

An analysis similar to Propositions~\ref{prop:PI} and~\ref{prop:PI_action} gives: 

\begin{proposition}\label{action_on_PII}
 There exists a holomorphic $\CC$-action on $\mathcal{M}^{PIIFN}$, equivariant with respect to the action $b\mapsto \llambda^4 b$ on $\mathcal{B}^{PIIFN}$. 
 It acts on $\Omega_I^{PIIFN}$ with weight $3$. 
\end{proposition}

\begin{proof}
 The Hitchin base is 
 \[
  \mathcal{B}^{PIIFN} = \{ ( z^{-5} + b z^{-3} ) (\d\! z)^{\otimes 2} \colon \quad b\in \C \},
 \]
see~\cite[Equation~(81)]{ISS3} (up to the relabelling $(z, \zeta, b) \mapsto (z_2, w_2, t)$ and the global factor $z_2^2$).
The action
\[
 z\mapsto \llambda^{-2} z, \quad \zeta \mapsto \llambda^3  \zeta 
\]
preserves $\mathcal{B}^{PIIFN}$ and acts on it by $b\mapsto \llambda^4 b$.
\end{proof}

Let us determine the $\CC$-fixed loci of $Y:=\M^{PII}$ and their weight decompositions; this does not depend on which of the two cases above we discuss, as in both cases the geometry of the $\CC$-action is the same.
Given the Hitchin fibration $\H: Y\fun \mathcal{B}^{PII}$, the core is (first equality is due to \cref{lemma_core_is_the_central_fibre})
$$\Core(Y)=h^{-1}(0)= \CP^1 \cup_{F_0} \CP^1$$
where two $\C P^1$ intersect each other in a tacnode (i.e., a tangency of second order, locally equivalent to $y^2 = x^4$) at a single point $F_0.$ Thus 
$\rk\, H^*(Y)=\rk\, H^*(\Core(Y))=3$.

As the fixed locus is smooth and its cohomology rank is equal to the cohomology of the total space, we either have 3 fixed points, or a fixed $\C P^1$ and a point on the other one. But the latter is impossible, as the two $\CP^1$ in the core share the tangent space at their intersection point, which is a weight$-0$ space, meaning the other $\C P^1$ would have to be fixed as well, contradiction.
Thus, the fixed locus 
$$Y^{\CC}=F_0 \sqcup F_1 \sqcup F_2$$
is the point of intersection $F_0$ of the two $\C P^1$'s and an another point on each of them, $F_1$ and $F_2.$
Assuming that $F_\min=F_1$, its weight decomposition is 
\begin{equation*} 
    T_{F_\min} Y = H_{k_1} \oplus H_{k_2},
\end{equation*}
where $k_1,k_2>0.$ Assume that e.g. $H_{k_2}$ corresponds to the $\CC$-flow in $\CP^1$, then it flows in $H_{-k_2}$-weight space in $F_0$, which is also tangent space of the other $\CP^1.$ Hence, flowing back within that $\C P^1$,
$H_{k_2}$-weight space appears in $F_2$ as well. But because of 
\begin{equation}\label{weights_of_the_minimum_PII}
    k_1+k_2=3
\end{equation}
condition (\cref{omega_C_pairing} and \cref{action_on_PII}), 
$F_2$ then has the same decomposition as $F_1,$ in particular it has all positive weights, giving another minimum point, which is impossible.\footnote{recall \cref{Morse_Bott_decomposition_theorem_for_moment_map}.}
Thus, we have
$$\F_\min=\F_0,$$
and its weight decomposition is (due to \eqref{weights_of_the_minimum_PII})
\begin{equation}\label{weight_decomp_minimum_PII}
    T_{F_0} Y = H_1 \oplus H_2.
\end{equation}
If $H_2$ corresponds to flow within $\C P^1$'s, then the weight decomposition of 
$F_1$ (and $F_2$) would have $H_{-2},$ thus (because of \cref{omega_C_pairing}) $H_5$ weight space as well.
The latter would then flow outside of the core, but $5 \nmid 4,$ contradiction.\footnote{recall that due to equivariance of $h:Y \fun \mathcal{B}^{PII}$ all outer weights have to divide the weight of the base $\mathcal{B}^{PII},$ which is due to \cref{action_on_PII} equal to 4.}
Thus, $H_1$ in \eqref{weight_decomp_minimum_PII} corresponds to the flow within the $\CP^1$'s, hence, due to \cref{omega_C_pairing} we have weight decompositions
\begin{equation}\label{weight_decomp_maximum_PII}
    T_{F_1} Y = H_{-1} \oplus H_{4},\ \ T_{F_2} Y = H_{-1} \oplus H_{4},
\end{equation}
and $H_2$ in \eqref{weight_decomp_minimum_PII} goes outside of the core.

It follows from~\cite[Lemma~2~(4)]{Sz_PW} that one of the two $\CP^1$ components of the core, say the one containing $F_2$, parametrizes flags at the logarithmic parabolic point $z=\infty$ for a Higgs bundle that has zero residue at that point, corresponding to the choice $b=0$ in the Hitchin base. 
This component therefore arises from blowing up the cuspidal singular point of the spectral curve lying over $z=\infty$. 
The $\CC$-flowline of weight $4$ emanating from the other fixed point $F_1$ is then the Hitchin section, just as for {\Pa}  I. 

Now, computation of the indices \eqref{Floer_grading_of_F_alphas_in_HF*(H_lambda)} gives:\footnote{where, as before we denote $\mu_{\lambda+}:=\mu_{\lambda+\delta}$ for $1\gg\delta>0$}
$$\mu_{1/4+}(F_0)=\mu_{1/4+}(F_1)=\mu_{1/4+}(F_2)=0,\ \mu_{1/2+}(F_0)=\mu_{1/2+}(F_1)=\mu_{1/2+}(F_2)=-2,$$
which together with $\mu_0=0, \mu_1=\mu_2=2$ and \cref{lower_bounds_filtration} gives the lower bounds
$$\rk\, \FF_{1/4}(H^2(Y)) \geq 2,\ \rk\, \FF_{1/2}(H^*(Y)) \geq 3,$$
hence $\FF_{1/4}\supset H^2(Y),$ and $\FF_{1/2}= H^*(Y)$.
Due to \cref{Cor intro about Fmin surviving} and \eqref{weight_decomp_minimum_PII}, $1\in H^0(F_{\min})=H^0(F_0)$ is not in the filtration until $\lambda=1/2,$
so, together with the previous, we have:
\begin{equation}\label{filtration_panleve_II}
    0\subset \FF_{1/4}=H^2(\M^{PII}) \subset \FF_{1/2} =H^*(\M^{PII}).
\end{equation}
These show precisely when the filtration jumps due to \cref{filtration_stability_theorem}, as the outer $S^1$-periods are multiples of $1/2$ and $1/4$.

We could also use the spectral sequence method to reach the same conclusion. 
From the weight decompositions 
\eqref{weight_decomp_minimum_PII},\eqref{weight_decomp_maximum_PII}
we get the outer torsion manifolds 
$$Y_4\iso \C \sqcup \C, \quad \ Y_2 \iso \C \sqcup \C \sqcup \C,$$ 
where one of the components of $Y_4$ is the 
Hitchin section, and the second 
and third components in $Y_2$ 
is $Y_4,$ as $\Z/2\leq \Z/4$.
Thus, their intersections with the energy hypersurfaces \eqref{Morse_Bott_submanifolds_of_orbits} are 
$$B_{1/4} \iso S^1 \sqcup S^1,\quad B_{1/2} \iso S^1 \sqcup S^1 \sqcup S^1.$$
Their shifts, computed via 
\eqref{CZindecesOfMorseBottSumbanifolds}, 
give the columns of the $E_1$-page of the {\MBF} spectral sequence, shown on \cref{MII_sp_seq}. 
As before, we show all differentials 
$\d_r,\ r\geq 1$ stapled together.

The column $H^*(B_1)$ is determined by \cref{hypersurface_cohomology}, as the intersection form of $\M^{PII}$
has a one-dimensional kernel, spanned by the sum of core components.
Indeed, denoting them by $E_1$ and $E_2$,
    $$(E_1+E_2)\cdot (E_1+E_2) =E_1^2 + 2E_1 \cdot E_2 + E_2^2 = -2 + 4 -2 =0,$$
    where each $E_i^2=-2$, as it is a Lagrangian curve, thus by the Weinstein neighborhood theorem, its normal bundle is isomorphic to its cotangent bundle, hence the self-intersection is the negative Euler characteristic, i.e. $-2$ ($E_i=\CP^1$). Intersection $E_1 \cdot E_2 = 2$ as they are tangential in order 2.
\begin{figure}[h]
		\centering
		{
			\includegraphics[scale=1]{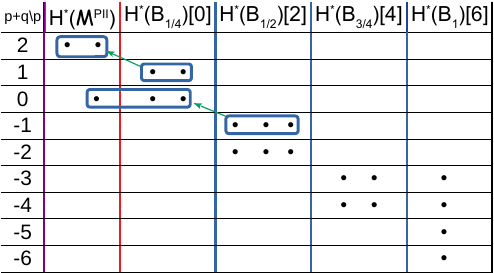}
                \caption{Spectral sequence for $X=II$}
                \label{MII_sp_seq}
		}
\end{figure}

As in the previous example, we conclude that the spectral sequence recovers the filtration \eqref{filtration_panleve_II}.
\subsection{Painlev\'e IV (degenerate) case}\label{sec Pa IV}

The divisor is $D = 3 \cdot \{ 0 \} + \{ \infty \}$. 
The local form of the Higgs field near $z=0$ is 
\[
   \theta = \left[ z^{-3} \begin{pmatrix}
         a_+ + b_+ z + \llambda_+ z^2 & 0 \\
         0 & a_- + b_- z + \llambda_- z^2 
        \end{pmatrix} + O(1) 
  \right] \otimes \d\! z. 
\]
At $z=\infty$ the Higgs field has a simple pole with non-semisimple residue. 
As before, we set 
\[
 a_{\pm} = \pm a, \quad b_{\pm} = \llambda_{\pm} = 0, 
\]
and then by the residue theorem, the residue at $z=\infty$ is nilpotent. 
According to \cite[Theorem~2.5]{ISS3} and \cite[Lemma~2~(5)]{Sz_PW}, this Hitchin fibration $\mathcal{M}^{PIV}$ has a unique singular fiber, that is of type $IV$, i.e. three genus-0 curves 
meeting transversally at a single point.
\begin{proposition}
 With the above assumptions, the trace of $\theta$ vanishes, and its determinant reads as 
 \[
  q(z) = ( -a^2 z^{-6} + b z^{-3} ) (\d\! z)^{\otimes 2} 
 \]
 for some $b\in \C$. 
 Conversely, for any $q$ of this form there exists $(\mathcal{E}, \theta)\in \mathcal{M}^{PIV}$ satisfying 
 \[
   \det ( \theta - \zeta ) = \zeta^2 + q(z). 
 \]
\end{proposition}

\begin{proof}
 This follows from~\cite[Equation~(70)]{ISS3}. 
 Note that in that formula, $z_2 = z^{-1}$, and we made use of the trivialization $\frac{\operatorname{d}\! z_2}{z_2}$ of $K(D)$ near $z=\infty$. 
 Furthermore, the parameter $b_{-1}$ vanishes, and our free parameter $b$ can be identified with the parameter $t$ appearing in~\cite{ISS3}. 
 Finally, $w_2 = \zeta$. 
\end{proof}
  
The space of quadratic differentials $q$ as in the Proposition will be called \emph{irregular Hitchin base} in the degenerate {\Pa} IV case, and denoted by $\mathcal{B}^{PIV}$. 
For any $b\in\C$ we define the curve $\Sigma^{PIV}_b$ by the equation 
\begin{equation}\label{eq:spectral_curve_PIV}
  \zeta^2 +  \left( -a^2 z^{-6} + b z^{-3} \right) (\d\! z)^{\otimes 2} = 0.
\end{equation}
The central fiber of type $IV$ lies over $b=0$.

\begin{proposition}\label{prop_C_action_IV}
 There exists a holomorphic $\CC$-action on $\mathcal{M}^{PIV}$, equivariant with respect to the action $b\mapsto \llambda^3 b$ on $\mathcal{B}^{PIV}$. 
  It acts on $\Omega_I^{PIV}$ with weight $2$. 
\end{proposition}

\begin{proof}
 Using the same idea and notation as in Propositions~\ref{prop:PI_action},~\ref{prop:PII_action}, we find 
  $w(z ) = -1,\ w(\zeta ) = 2, $
 and we check that this action is equivariant with respect to 
  $b\mapsto \llambda^3 b. $ 
\end{proof}
Let us determine the $\CC$-fixed locus of $Y:=\M^{PIV}$ and their weight decompositions; 
Given the Hitchin fibration $\H: Y\fun \mathcal{B}^{PIV}$, the core is 
$$\Core(Y)=h^{-1}(0)= \CP^1 \cup_{F_0} \CP^1 \cup_{F_0} \CP^1$$
where three $\C P^1$ intersect transversely at a single point $F_0.$ Thus
$\rk\, H^*(Y)=\rk\, H^*(\Core(Y))=4$.

As the fixed locus is smooth and its cohomology rank is equal to the cohomology of the total space, we either have 4 fixed points, or a fixed $\C P^1$ and a point on the other two. But the latter is impossible, as the weight decomposition of $F_0$ would then be $T_{F_0}Y= H_0 \oplus H_2$,
but the tangent spaces of the other two $\CP^1$ are $\CC$-invariant subspaces, thus both have to be equal to $H_2$, which is a contradiction as they are transverse.
Thus, the fixed locus 
$$Y^{\CC}=F_0 \sqcup F_1 \sqcup F_2 \sqcup F_3$$
is the point of intersection $F_0$ of the three $\C P^1$'s and another point on each of them, $F_1, F_2$ and $F_3$.
Because of the argument above, the weight decomposition of $F_0$ has to be homogeneous, hence due to \cref{omega_C_pairing} and the last part of \cref{prop_C_action_IV}
it is 
\begin{equation}\label{weight_decomp_minimum_PIV}
   T_{F_0} Y=H_1, 
\end{equation}
in particular 
$F_\min=F_0$.
Therefore, subspaces of $H_1$ corresponds to flow within $\C P^1$'s, 
thus the weight decomposition of 
$F_1, F_2$ and $F_3$ has $H_{-1},$ thus (due to \cref{omega_C_pairing}) $H_3$ weight space:
\begin{equation}\label{weight_decomp_maximum_PIV}
    T_{F_1} Y = H_{-1} \oplus H_{3},\quad T_{F_2} Y = H_{-1} \oplus H_{3},\quad
    T_{F_3} Y = H_{-1} \oplus H_{3},
\end{equation}
which 
goes outside of the core. 
As before, the $\CC$-flowline corresponding to $H_{3}$ from one of these fixed points is the Hitchin section. 
Now, computation of the indices \eqref{Floer_grading_of_F_alphas_in_HF*(H_lambda)} gives:\footnote{where, as before we denote $\mu_{\lambda+}:=\mu_{\lambda+\delta}$ for $1\gg\delta>0$}
$$\mu_{1/3+}(F_0)=\mu_{1/3+}(F_1)=\mu_{1/3+}(F_2)=\mu_{1/3+}(F_3)=0,$$
which together with \cref{lower_bounds_filtration} 
gives the lower bound
$\rk\, \FF_{1/3}(H^2(Y)) \geq 3,$ 
hence $\FF_{1/3}\supset H^2(Y).$ 
\cref{unit entering for weight-1 and weight-2 SHS} tells us that $1\in H^0(Y)$ enters the filtration at $\lambda=1$, hence altogether:
\begin{equation}\label{filtration_panleve_IV}
    0\subset \FF_{1/3}=H^2(\M^{PIV}) \subset \FF_{1} =H^*(\M^{PIV}),
\end{equation}
which are precisely filtration jumps due to \cref{filtration_stability_theorem}, as the outer $S^1$-periods are multiples of $1/3$. 

Let us discuss the spectral sequence method. 
From the weight decompositions 
\eqref{weight_decomp_minimum_PIV},\eqref{weight_decomp_maximum_PIV}
we get the outer torsion manifolds 
$$Y_3\iso \C \sqcup \C \sqcup \C.$$ 
Their intersections with the energy hypersurfaces \eqref{Morse_Bott_submanifolds_of_orbits} are 
$B_{1/3} \iso S^1 \sqcup S^1 \sqcup S^1.$
Their shifts, computed via \eqref{CZindecesOfMorseBottSumbanifolds},
give the columns of the $E_1$-page of the {\MBF} spectral sequence, shown on \cref{MIV_sp_seq}, where, as before, all differentials are stapled together.

The column $H^*(B_1)$ is determined by \cref{hypersurface_cohomology}, as the intersection form of $\M^{PIV}$
has a one-dimensional kernel, spanned by the sum of core components (labeled by $E_{1,2,3}$):
    $$(E_1+E_2+E_3)\cdot (E_1+E_2+E_3) =\sum_i E_i^2 + 2 \sum_{i<j} E_i \cdot E_j
    = 3(-2) + 3 \cdot 2 =0,$$
    where $E_i^2=-2$ as in PII example, and $E_i \cdot E_j =1$ as they intersect transversally in a point. 

\begin{figure}[h]
		\centering
		{
			\includegraphics[scale=1]{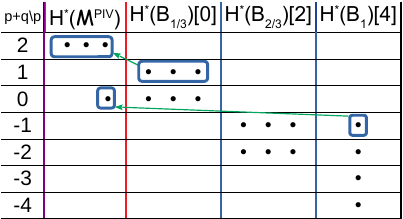}
                \caption{Spectral sequence for $X=IV$. 
                }
                \label{MIV_sp_seq}
		}
\end{figure}
Unlike the previous examples, the spectral sequence does not completely determine the filtration \eqref{filtration_panleve_IV} --
one could not tell a priori whether $1\in H^0(\M^{PIV})$ on \cref{MIV_sp_seq} is annihilated by column $B_{2/3}$ or by column $B_1$.

\subsection{Painlev\'e VI case}\label{sec:PVI}

Let $D = \{ 0 \} + \{ 1 \} + \{ t \} + \{ \infty \}$ for some fixed $t\in \C P^1 - \{ 0, 1, \infty \}$. 
We consider Higgs fields $\theta$ with a simple pole at $D$, with nilpotent residue.  
The corresponding Hitchin moduli space $\mathcal{M}^{PVI}$ has a unique singular fiber, of type $\widetilde{D}_4 = I_0^*$, i.e. affine Dynkin $\widetilde{D}_4$ tree of genus 0 curves 
interesting transversally.
For any such Higgs field, we have 
\[
 \det (\theta ) = b \frac{(\d\! z)^{\otimes 2}}{z(z-1)(z-t)} 
\]
for some $b\in\C$, see~\cite[Section~5]{Sz_PW}. 
We let $\mathcal{B}^{PVI}$ stand for the space of quadratic differentials of this form.

\begin{proposition}
 There exists actions of $\C^*$ on $\mathcal{M}^{PVI}$, equivariant with respect to the action $b\mapsto \llambda^2 b$ on $\mathcal{B}^{PVI}$. 
 It acts on $\Omega_I^{PVI}$ with weight $1$. 
\end{proposition}

\begin{proof}
 Here we must have $w(z) = 0$, for otherwise the positions of the poles $1,t$ would change. 
 We may set $w(\zeta ) = 1$, and then we easily find that $b$ transforms as $b\mapsto \llambda^2 b$.
\end{proof}

As the core of $Y:=\M^{PVI}$ is equal to the central fiber of $h: \M^{PVI}\fun \mathcal{B}^{PVI},$
$$\Core(Y)=h^{-1}(0)=\{ \widetilde{D}_4 \text{ tree of } {\CP^1}\text{s} \},$$
the $\CC$-fixed locus has to have a total cohomology rank of 6. The $\CP^1$ corresponding to the central vertex has 4 fixed points, hence it is fixed.
Due to the smoothness of fixed loci, no other $\CP^1$s are fixed, hence they have another fixed point each, thus:
$$Y^{\CC}=\{F_0=\CP^1 \} \sqcup F_1 \sqcup \dots \sqcup F_4.$$
The weight decomposition of $F_0$ is (due to \cref{omega_C_pairing})
\begin{equation}\label{weight_decomposition_minimum_PVI}
    T_{F_0} Y = H_0 \oplus H_1,
\end{equation}
in particular $\F_\min=F_0.$ Using the $\CC$-flow and \cref{omega_C_pairing} again, one concludes
\begin{equation}\label{weight_decomposition_maximum_PVI}
T_{F_i} Y = H_{-1} \oplus H_2,\ i=1,\dots, 4, 
\end{equation}
and as before, the $\CC$-flowline corresponding to $H_2$ from one of
these fixed points is the Hitchin section. 
Computing the indices \eqref{Floer_grading_of_F_alphas_in_HF*(H_lambda)} and via \cref{lower_bounds_filtration}, one gets
$$\rk\, \FF_{1/2}(H^2(Y)) \geq 4,\ \rk\, \FF_{1}(H^2(Y)) \geq 5,$$
which gives $\FF_1 \supset H^2(Y).$
Together with \cref{Cor intro about Fmin surviving} we have
$$\rk\, \FF_{1/2} = 4, \ \FF_{\B,1/2}= \oplus_{i=1}^4 H^0(F_i)[-2],$$
and on the other hand $H^0(F_i)[-2]=\la [E_i] \ra$ in $H^2(Y)$,
where $E_i$ is the $\C P^1$ which $F_i$ belongs to.
By \cref{unit entering for weight-1 and weight-2 SHS}, $1\in H^0(\F_0)$
enters the filtration at $\lambda=2,$ hence altogether we have:
\begin{align}
\begin{split}
\label{Filtration_Painleve_VI}
    \FF_{1/2} \iso \k^4 \subset \FF_1=H^2(Y;\k) \subset \FF_2= H^*(Y;\k) \\
    \FF_{\B, 1/2} = \la [E_1],\dots,[E_4] \ra \subset \FF_{\B,1}=H^2(Y) \subset \FF_2= H^*(Y)
\end{split}
\end{align}

From the weight decompositions 
\eqref{weight_decomposition_minimum_PVI},\eqref{weight_decomposition_maximum_PVI}, we get the outer torsion manifolds 
$Y_2\iso \C \sqcup \C \sqcup \C \sqcup \C,$ 
 thus 
$B_{1/2} \iso S^1 \sqcup S^1 \sqcup S^1 \sqcup S^1.$
Their shifts, computed via \eqref{CZindecesOfMorseBottSumbanifolds}, give the columns of the $E_1$-page of the spectral sequence, shown on \cref{MVI_sp_seq}. 
The column $H^*(B_1)$ is determined by \cref{hypersurface_cohomology}, as the intersection form of $\M^{PVI}$
has a one-dimensional kernel, spanned by the weighted sum $2[F_0] + [E_i]$ of the core components.
Like in the previous example, the spectral sequence alone does not completely determine the filtration \eqref{Filtration_Painleve_VI},
as we could not tell which of the columns $B_{1}, B_{3/2}$ and $B_2$ annihilates
$1\in H^0(\M^{PIV})$.
\begin{figure}[h]
		\centering
		{
			\includegraphics[scale=1]{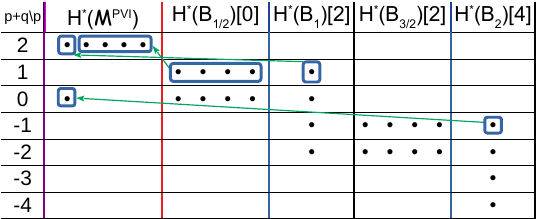}
                \caption{Spectral sequence for $X=VI$}
                \label{MVI_sp_seq}
		}
\end{figure}

\subsection{Remaining Painlev\'e cases}\label{sec:PIII-V}

\begin{proposition}\label{prop:existence}
 Let $\mathcal{M}^{PX}$ be any {\Pa} moduli space in the Dolbeault complex structure, with $X\in \{ III, V \},$ or non-degenerate versions of $II, IV, VI$ (i.e., ones with regular semi-simple residue at (some of) the simple pole(s)). 
 Then there exists no $\C^*$-action on $\mathcal{M}^{PX}$ equivariant with respect to the Hitchin fibration. 
\end{proposition}

\begin{proof}
 The existence of an equivariant action implies that the corresponding fibration has at most one singular fiber (otherwise, all fibers would be singular). 
 The statement now follows from~\cite[Theorems~1.1--1.6]{ISS2}. 
\end{proof}

\subsection{Further Higgs moduli of $\dim_\C=2$}\label{sec:E6}

A rank $3$ parabolic Hitchin moduli space of complex dimension $2$ is discussed in detail in~\cite[Chapter~2]{Mihajlovic_thesis}. 
Namely, $D = \{ 0 \} + \{ 1 \} + \{ \infty \}$ and the residue of $\theta$ has a triple eigenvalue $0$ at each logarithmic point. 
The corresponding moduli space $\mathcal{M}^{\widetilde{E}_6}$ has a unique singular fiber (under Hitchin fibration), 
which is an affine $\widetilde{E}_6$ tree of smooth rational curves.

This example together with {\Pa} VI are just two out of the five $2$-dimensional moduli spaces of parabolic Higgs bundles described in~\cite{groechenig2014hilbert} (see also~\cite{zhang2017multiplicativity}) as crepant resolutions $$\pi: \M_{\Gamma} \fun (T^* E)/\Gamma$$ for an elliptic curve $E$ and a finite subgroup $\Gamma \in \{0, \Z/2, \Z/3, \Z/4, \Z/6 \}$ of isometries of $E$. To these, one associates the affine
Coxeter--Dynkin graphs $Q_\Gamma= \widetilde{A}_0, \widetilde{D}_4, \widetilde{E}_6, \widetilde{E}_7, \widetilde{E}_8$ respectively.
The projection 
\begin{equation}\label{projection_A_1_affine}
    p: T^*E \iso E \times \C \fun \C
\end{equation}
together with $\pi$ induces Hitchin fibration 
\begin{equation}\label{Hitchin_fibration_for_M_gamma}
    h_{\Gamma}: \M_\Gamma \fun  (T^*E)/\Gamma \fun  \C/\Gamma=\mathcal{B}_\Gamma \iso \C.
\end{equation}
\begin{prop}\label{prop:Zhang}\cite[Proposition 5.4]{zhang2017multiplicativity}
The central fiber of \eqref{Hitchin_fibration_for_M_gamma} is
    $$\Core(\M_\Gamma)=h_{\Gamma}^{-1}(0)=\{Q_\Gamma \text{-tree of curves} \},$$
    where the curves are rational, except the $\Gamma=\{0\}$ case where the curve is elliptic.
\end{prop}

The fiber-contraction on $T^*E$
that makes \eqref{projection_A_1_affine} equivariant (with weight-1 action on $\C$) 
induces a $\CC$-action on $(T^*E)/\Gamma,$ which lifts on $\M_\Gamma$ 
as it is a crepant resolution of a ``nice'' isolated singularity 
(one can use e.g. \cite[Proposition 3.6.]{McLR18}).
This altogether makes \eqref{Hitchin_fibration_for_M_gamma} equivariant.
For completeness, we describe this action in the Higgs bundle-theoretic way:
\begin{proposition}
For all $m\in \{ 2,3,4,6 \}$ and $\Gamma = \Z/m$, there exists a $\CC$-action on $\M_\Gamma$, equivariant with $b\mapsto \llambda^m b$ on $\mathcal{B}_\Gamma.$ 
 The weight of this action on the holomorphic symplectic form $\Omega_I$ 
 is $1$. 
\end{proposition}

\begin{proof}
For $m=2$, see Section~\ref{sec:PVI}. 
In the case $m=3$ $(i.e., \widetilde{E}_6)$, the characteristic polynomial of $\theta$ reads as 
\[
 \zeta^3 + b \frac{(\d\! z)^{\otimes 3}}{z^2(z-1)^2} 
\]
for some $b\in\C$, see~\cite[Chapter~2]{Mihajlovic_thesis}. The action is 
$$z  \mapsto z,\ \zeta \mapsto \llambda \zeta,\ b  \mapsto \llambda^3 b$$

A similar argument works in the other cases too: for $m=4$, the Hitchin base is 
\[
    H^0 (\C P^1, K^{\otimes 4} (3 \cdot \{ 0 \} + 3 \cdot \{ 1 \} + 2 \cdot \{ \infty \})) \iso \C, 
\]
and the equation of the spectral curve is 
\[
 \zeta^4 + b \frac{(\d\! z)^{\otimes 4}}{z^3(z-1)^3}=0. 
\]
The equivariant action is given by 
$z  \mapsto z,\ \zeta \mapsto \llambda \zeta,\ b  \mapsto \llambda^4 b.$

Finally, for $m=6$, the Hitchin base is 
\[
    H^0 (\C P^1, K^{\otimes 6} (5 \cdot \{ 0 \} + 4 \cdot \{ 1 \} + 3 \cdot \{ \infty \})) \iso \C, 
\]
 the equation of the spectral curve is 
\[
 \zeta^6 + b \frac{(\d\! z)^{\otimes 6}}{z^5(z-1)^4}=0, 
\]
and the equivariant action is given by
$z  \mapsto z,\ \zeta \mapsto \llambda \zeta,\ b  \mapsto \llambda^6 b.$
%
\end{proof}

The Floer-theoretic filtration $\FF$ with respect to this $\CC$-action on $\M_\Gamma$ is studied in \cite{RZ2}, where it is compared with the perverse filtration of Hitchin fibration
$\H_\Gamma: \M_\Gamma \fun \mathcal{B}_\Gamma$
(the famous ``$P=W$'' filtration of Higgs moduli spaces). In addition, there is an interesting correlation with the associated root system of graph $Q_\Gamma.$ Being an affine Dynkin graph, it has the \textit{imaginary root}, the positive $\Z$-generator of the 
radical of the symmetric bilinear form with associated Cartan matrix $C_{Q_\Gamma}.$
One can label the vertices $i\in Q_\Gamma^0$ by the corresponding
entries $n_i$ of the imaginary root, see
\cref{E_8_imaginary_root}.
 \begin{figure}[h]%
				\centering
				{
     \includegraphics[scale=0.13]{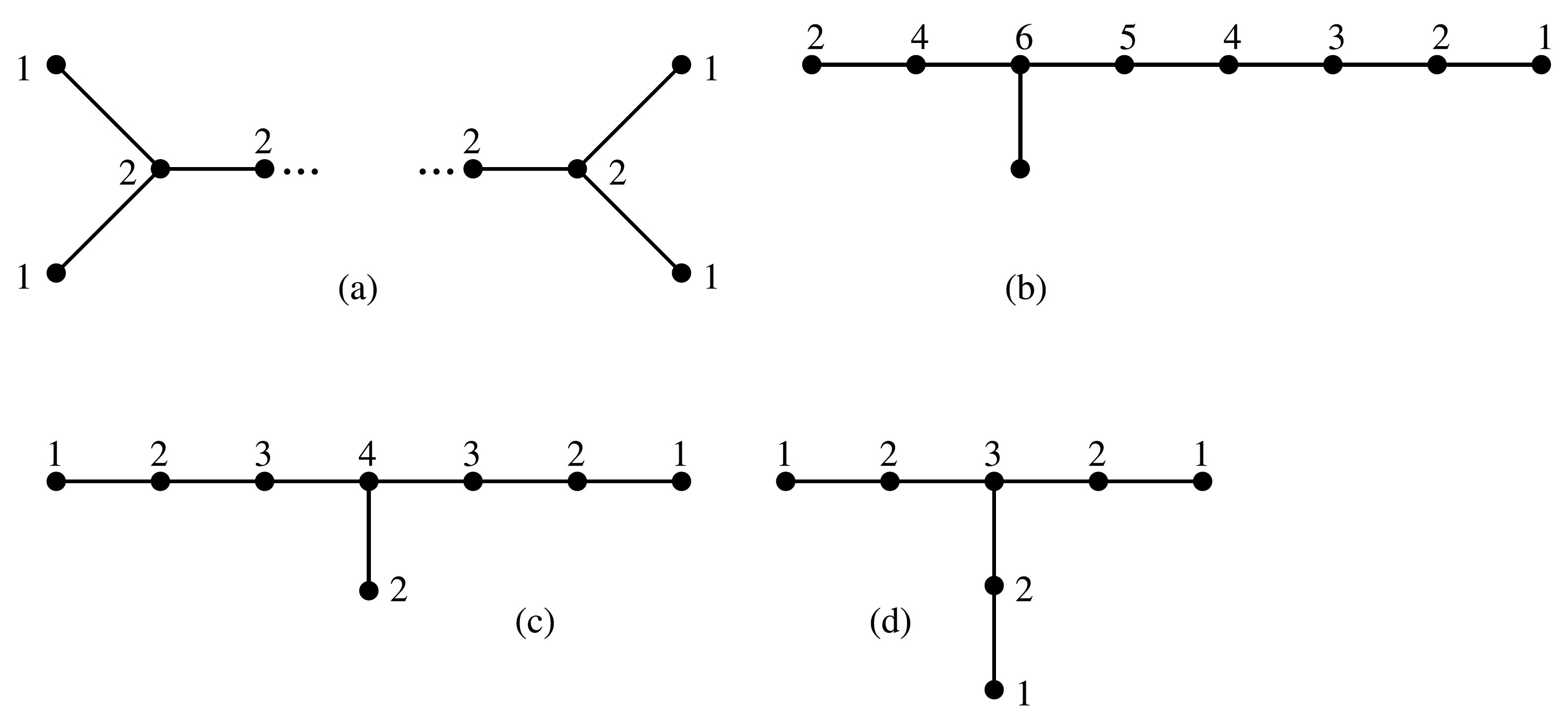}
					\caption{Imaginary root labellings for (a) $\widetilde{D}_n$ (total number of vertices is equal to $n+1$)
     (b) $\widetilde{E}_8$ (c) $\widetilde{E}_7$ (d) $\widetilde{E}_6$.}
					\label{E_8_imaginary_root}
				}
			\end{figure}

Given two filtrations $\{F_i\}_i$ and $\{G_j\}_j$ of a vector space $V,$ we say that $F$ \textbf{refines} $G$ if for each $j$ there is some $i$ such that $G_j=F_i$.

\begin{prop} \cite[Example~7.23]{RZ2} \label{Floer filtration on parabolic 2-dim equals imaginary root labelling}
     $\FF:=\FF_{\B}(H^2(\M_\Gamma))$ refines the perverse filtration of 
     $h_\Gamma: \M_\Gamma\fun\mathcal{B}_\Gamma.$
     Moreover, its ranks 
    correlate with the labels of the imaginary root of the
    root system of $Q_\Gamma$:
     $$\rk\, \FF_{k/|\Gamma|} = \#\{i \in Q_\Gamma^0 \mid n_i\leq k\}$$
\end{prop}

\begin{remark}
    The actual result from \cite[Example~7.23]{RZ2} is slightly less informative -- it mentions the \textit{$k$-th step of the filtration} 
    instead of its precise $\lambda=k/|\Gamma|,$ but 
    the latter value is easy to conclude by looking at the spectral sequences, 
    which are used to compute filtrations $\FF$ in the first place.
\end{remark}

\subsection{Classification of 2-dimensional $\C^*$-equivariant Higgs moduli spaces}

\begin{prop}\label{prop:classification}
    Let $\mathcal{M}$ be a moduli space of Higgs bundles over some curve $X$. 
    Assume that
    \begin{enumerate}
        \item\label{item:1} $\mathcal{M}$ is of complex dimension $2$; 
        \item\label{item:2} there exists a smooth Hitchin map $h\colon \mathcal{M}\to \C$ turning $\mathcal{M}$ into a smooth elliptic fibration; 
        \item\label{item:3} $\C^*$ acts holomorphically on $\mathcal{M}$ and $h$ is equivariant for some contracting action on $\mathbb{C}.$
    \end{enumerate}
    Then, $\mathcal{M}$ is biregular to one of the spaces $\mathcal{M}^{PX}$ or $\M_\Gamma$ described in Sections~\ref{sec:PI}--\ref{sec:E6}. 
    Moreover, $h$ agrees with the corresponding Hitchin map of $\mathcal{M}^{PX}$ and $\M_\Gamma$ too. 
\end{prop}
\begin{proof}
    It follows from assumption~\eqref{item:2} that $\mathcal{M}$ is a Zariski open subset of the rational elliptic surface $\operatorname{Bl}_9 \C P^2$ ($\C P^2$ blown up $9$ times). 
    The Hitchin map is algebraic, therefore it extends to an elliptic fibration 
    \[
        \tilde{h}\colon \operatorname{Bl}_9 \C P^2 \to \C P^1. 
    \]
    From assumption~\eqref{item:3} and smoothness of $h$ we infer that $\tilde{h}$ may have at most two singular fibers, namely the fibers over $0,\infty \in \C P^1$. 

    If $h$ admits no singular fiber, then $\M$ is diffeomorphic to $T^* E$ for an elliptic curve $E$, i.e. $\M_\Gamma$ for $\Gamma = 0$. 
    
    Assume that $h$ admits exactly one singular fiber, and $\tilde{h}$ has exactly two singular fibers.  
    Then, from the structure of the monodromy of the Kodaira fibers, it is well-known that the only possibilities for couples of singular fibers over $\operatorname{Bl}_9 \C P^2$  are $I_0^* + I_0^*$, $\widetilde{E}_6 + IV$, $\widetilde{E}_7 + III$ or $\widetilde{E}_8 + II$. 
    
    The case $I_0^* + I_0^*$ is symmetric. 
    In the other cases, the fiber over $0$ can be either of the two singular fibers. 
    In Sections~\ref{sec:PI}--\ref{sec:E6} we have found all these possibilities. 
\end{proof}
\begin{remark}
    We would like to emphasize that there exist different moduli problems from the ones detailed here that have diffeomorphic moduli spaces. 
    Namely, in 
    \cite[Section~11]{boalch2012simply}
    different "readings" of a quiver are introduced and studied.
    For instance, in \cite[Section~11.3]{boalch2012simply}, infinitely many Lax representations of $PIV$ are described. 
    It would be interesting to work out in detail the Hitchin map for such other spaces explicitly. 
    We leave this question for further study. 
    The main challenge is that the spectral correspondence produces many-sheeted covers of $\C P^1$, which then get birationally transformed to high degree (and thus high genus) curves in $\C P^2$, whose Jacobi varieties are correspondingly of high dimension. 
    Therefore, in general it is not clear to the authors how they give rise to elliptic surfaces.
    However, for the rank $3$ Lax pairs of the Painlev\'e equations,  in~\cite{ESz} it is proven that Fourier--Laplace transformation establishes symplectic isomorphisms with the Hitchin moduli spaces treated here. 
\end{remark}

\section{Three filtrations compared}\label{P=W_Floer_Multiplicity}

The imaginary root labels mentioned above have another geometric meaning.
Recall from ~\cref{prop:Zhang} that 
$\Core(\M_{\Gamma})=h_{\Gamma}^{-1}(0)=\cup E_i$
has irreducible components $E_i$ labeled by vertices $i$ of $Q_\Gamma.$ 
Moreover, its canonical scheme structure comes with multiplicities $m_i$ of curves $E_i$
$$h_{\Gamma}^{-1}(0)= \sum_{i \in Q_\Gamma^0} m_i E_i$$
An intrinsic definition of the $m_i$ is that the Weil divisor on the right-hand side is linearly equivalent to the generic fiber of $h_{\Gamma}$ (that is a smooth abelian variety). 
An interesting phenomenon 
is that these coincide
\begin{equation}\label{imaginary root reads the multiplicities}
    m_i = n_i,
\end{equation}
with the labelings $n_i$ of the imaginary root of $Q_\Gamma.$ This is
a consequence of 
\begin{enumerate}
    \item Self-intersection of $h_{\Gamma}^{-1}(0)$ (with multiplicities) in the sublattice spanned by $\Irr(\Core(\M_\Gamma))$ is equal to zero, as a fiber class.
    \item By definition, the imaginary root is the positive integer generator of the radical of the symmetric bilinear form with given Cartan matrix, which equals minus the intersection matrix of the elements $E_i$ of $\Irr(\Core(\M))$.
\end{enumerate} 

Hence, we can now rephrase 
\cref{Floer filtration on parabolic 2-dim equals imaginary root labelling} 
as 
\begin{equation}
\label{floer filter by irr comp}
    \rk\, \FF_{k/|\Gamma|} = \#\{E_i \in \Irr(\Core(\M_\Gamma)) \mid m_i\leq k\}
\end{equation}

This is an interesting phenomenon, motivating us to define another filtration.
Consider any Higgs moduli space $\M$ with a $I$-holomorphic $\CC$-action,
and its Hitchin fibration $\H:\M\fun \mathcal{B}\iso \C^{\frac{1}{2} \dim_\C \M}$. As above, its core together with its inherited scheme structure has irreducible components with multiplicities 
$$\Core(\M)=h^{-1}(0)=\sum_i m_i E_i,\ \ m:=\max_i\, m_i$$
%
\begin{de}
Define the \textbf{multiplicity filtration} on $H^{\dim_\C \M}(\M)$ to be
    $$\MM_{\lambda}:=\la [E_i] \mid m_i/m \leq \lambda \ra.$$ 
Obviously, $\MM_1=H^{\dim_\C \M}(\M).$
\end{de}
Thus, \cref{Floer filtration on parabolic 2-dim equals imaginary root labelling} reads now as

\begin{corollary}
    For any 2-dimensional parabolic Higgs bundle $\M_\Gamma,$ we have
\begin{equation} \label{parabolic Higgs floer equals multiplicity}
    (\forall \lambda)\  \rk\, \FF_\lambda = \rk\, \MM_\lambda,
\end{equation}
\end{corollary}
\begin{proof}
    Notice by \eqref{imaginary root reads the multiplicities} that the maximal multiplicity $m$ of the irreducible core components 
    corresponds to the maximal labeled number of the associated imaginary root. 
    Looking at \cref{E_8_imaginary_root}, we see that this maximal root label is precisely $|\Gamma|$, as for
    $\widetilde{D}_4, \widetilde{E}_6, \widetilde{E}_7, \widetilde{E}_8$,
    the corresponding  groups $\Gamma$ are (as mentioned before) 
    $\Z/2, \Z/3, \Z/4, \Z/6,$ respectively. For  $\widetilde{A}_0,$ the multiplicity of the (only) component is $1,$ so again equal to $|\Gamma|$, as $\Gamma=0$ in that case.
    Thus, $m=|\Gamma|$, and by \eqref{floer filter by irr comp}
    $$\rk\, \FF_{k/\Gamma} = \#\{E_i \mid m_i\leq k\} = 
    \rk\,  \la [E_i] \mid m_i/|\Gamma|\leq k/|\Gamma| \ra = \la [E_i] \mid m_i/m\leq k/|\Gamma| \ra  = \rk\, \MM_{k/\Gamma},$$
    and both filtrations only jump in rank for $\lambda=k/\Gamma,$ hence \eqref{parabolic Higgs floer equals multiplicity} is proved.
\end{proof}
Following ideas from \cite{RZ1}, one expects that these filtrations actually coincide, not only rank-wise:

\begin{con} For any 2-dimensional parabolic Higgs bundle $\M_\Gamma,$ the Floer-theoretic and multiplicity filtration coincide:
$$\FF_\lambda =\MM_\lambda.$$
\end{con}

In summary, together with the results from \cref{sec:Painleve}, we have the following:

\begin{thm} Let us be given any Higgs moduli space $\M$ satisfying the conditions of Proposition~\ref{prop:classification}. 
Then, its multiplicity and Floer-theoretic filtration coincide rank-wise
   $$(\forall \lambda)\  \rk\, \FF_\lambda = \rk\, \MM_\lambda.$$
For moduli spaces of parabolic Higgs bundles, they both refine the {\PW} filtration, whereas 
for {\Pa} spaces, {\PW} refines the former two filtrations. 
Thus, in the intersection, i.e. for {\Pa} VI space, these three filtrations coincide. 
\end{thm}
\begin{proof}
    By Proposition~\ref{prop:classification}, we only need to deal with the cases covered in Sections~\ref{sec:PI}--\ref{sec:E6}. 
    By \eqref{parabolic Higgs floer equals multiplicity}, the multiplicity and Floer-theoretic filtration coincide rank-wise for parabolic spaces $\M_\Gamma.$ 
    For {\Pa}  spaces $\M^{PX},\ X\in \{I, II, IV\}$, both filtrations are trivial, so they coincide.
    The Floer-theoretic  filtration $\FF=\FF_\B(H^2(\M^{PX}))$ is trivial due to \eqref{filtration_panleve_I},\eqref{filtration_panleve_II},\eqref{filtration_panleve_IV}, respectively.
    The multiplicity filtration is trivial as the multiplicities of core components $E_i$ are all $m_i=1$ 
    in these cases. 
    The reasoning for this is similar to that for \eqref{imaginary root reads the multiplicities}, as the self-intersection of 
    $\sum_i E_i$ is zero. The latter has already been proved at the 
    ends of their respective \cref{sec:PI,sec Pa II,sec Pa IV}.
    
    For the parabolic moduli spaces $\M_\Gamma$, the multiplicity filtration $\MM$ refines the {\PW}, as for $\Gamma\neq 0,$ the latter has two 
    steps (see \cite[Proposition 5.4]{zhang2017multiplicativity}):
    \begin{equation}\label{P=W on H^2 for parabolic 2-dim Higgs}
        P_0=0\subset P_1= \la [E_i] \mid i \text{ not the central vertex of } Q_\Gamma \ra \subset 
P_2= H^2(\M_\Gamma),
    \end{equation}
    whereas the central vertex has the biggest multiplicity (see \cref{E_8_imaginary_root}), thus $P_1$ 
    is the penultimate filtered subspace of $\MM$. For $\Gamma=0,$ i.e. $\M=T^*E$, 
    both filtrations are trivial. 

    For {\Pa} I, II, and IV spaces, $\MM$ and $\FF$ are both trivial, hence {\PW} refines them. This is a strict refinement for {\Pa} II and IV, as there {\PW} has two steps, see \cite[Table 2.]{Sz_PW}.
    
    Finally, for 
    {\Pa} VI space , i.e. $\M_{\Z/2},$ {\PW} filtration, on \eqref{P=W on H^2 for parabolic 2-dim Higgs}, coincides with
    $\FF$, on \eqref{Filtration_Painleve_VI}, and $\MM$, as multiplicities are 2 for the central vertex and 1 for the others, see \cref{E_8_imaginary_root}(a).
\end{proof}

Thus, the last theorem compares the three filtrations for all 2-dimensional Higgs moduli
with $\CC$-actions, i.e. for which filtrations $\FF$ and $\MM$ make sense.
An interesting question is what happens in higher-dimensional examples? Due to \cite{groechenig2014hilbert}, we know that the Hilbert schemes of parabolic 2-dimensional moduli, 
$\Hilb^n(\M_\Gamma)$, are parabolic Higgs moduli spaces as well.
It would be interesting to understand the interplay between  {\PW}, Floer-theoretic, 
and multiplicity filtration, for these spaces.
\bibliography{FZ}
\bibliographystyle{amsalpha}
\end{document}